%% file: Lifting_of_recollements_and_Gorenstein_projective_modules.tex
\UseRawInputEncoding
\documentclass[oneside, a4paper,reqno]{amsart}
\usepackage{pdfsync}
\usepackage{stmaryrd}
\usepackage{mathrsfs}
\usepackage{multicol}
\usepackage{amsmath, amsthm, amscd, amssymb, latexsym, eucal}
\usepackage[all]{xy}
\def\serieslogo@{} \def\@setcopyright{} \makeatother
\usepackage{multienum}

\usepackage{hyperref}
\usepackage{color}
\usepackage{cite}

\hypersetup{colorlinks=true,
     breaklinks=true,
     linkcolor=blue,    % how to specify linkcolor according to chapter / contents ?
     bookmarks=true,
      pageanchor=true}

\makeatletter
\renewcommand*\env@matrix[1][c]{\hskip -\arraycolsep
  \let\@ifnextchar\new@ifnextchar
  \array{*\c@MaxMatrixCols #1}}
\makeatother

\usepackage{color}

\input{rgb}

\usepackage[colorinlistoftodos]{todonotes}

\numberwithin{equation}{section}
\newtheorem{thm}{Theorem}[section]
\newtheorem{cor}[thm]{Corollary}
\newtheorem{lem}[thm]{Lemma}
\newtheorem{prop}[thm]{Proposition}

\theoremstyle{definition}
\newtheorem{defn}[thm]{Definition}

\newtheorem{exam}[thm]{Example}

\newcommand{\lxr}{\longrightarrow}
\newcommand{\iso}{\cong}
\newcommand{\mr}{\mathsf{r}}
\newcommand{\mq}{\mathsf{q}}
\newcommand{\mi}{\mathsf{i}}
\newcommand{\ml}{\mathsf{l}}
\newcommand{\me}{\mathsf{e}}
\newcommand{\map}{\mathsf{p}}
\newcommand{\X}{\mathcal X}
\newcommand{\Y}{\mathcal Y}
\newcommand{\Z}{\mathcal Z}
\newcommand{\D}{\mathcal D}
\newcommand{\cS}{\mathcal S}
\newcommand{\U}{\mathcal U}
\newcommand{\V}{\mathcal V}
\newcommand{\W}{\mathcal W}
\newcommand{\F}{\mathscr F}
\DeclareMathOperator*{\im}{\mathsf{Im}}
\DeclareMathOperator*{\Ker}{\mathsf{Ker}}

\DeclareMathOperator*{\Gproj}{\mathsf{Gproj}}

\DeclareMathOperator*{\Loc}{\mathsf{Loc}}
\DeclareMathOperator*{\Ab}{\mathsf{Ab}}
\DeclareMathOperator*{\cht}{\mathsf{ht}}
\DeclareMathOperator*{\pGp}{\mathsf{pGp}}
\DeclareMathOperator*{\pd}{\mathsf{pd}}
\DeclareMathOperator*{\Mor}{\mathsf{Mor}}
\DeclareMathOperator*{\CS}{\mathsf{S}}
\newcommand{\La}{\Lambda}

\newsavebox{\proofbox}
\savebox{\proofbox}{\begin{picture}(7,7)%
  \put(0,0){\framebox(7,7){}}\end{picture}}

\begin{document}
%\linenumbers

\title[]{Lifting of recollements and Gorenstein projective modules}
\author[]{Nan Gao$^{*}$, Jing Ma
}
\address{Department of Mathematics, Shanghai University, Shanghai 200444, PR China}
\thanks{* is the corresponding author.}
\email{nangao@shu.edu.cn, majingmmmm@shu.edu.cn  }
\thanks{Supported by the National Natural Science Foundation of China (Grant No. 11771272 and 11871326).}

\date{\today}



\begin{abstract}\ In the paper, we investigate the lifting of recollements with respect to Gorenstein-projective modules. Specifically,
a homological ring epimorphism can induce a lifting of the recollement of the stable category of finitely generated Gorenstein-projective modules; the recollement of the bounded Gorenstein derived categories of some upper triangular matrix algebras can be lifted to the homotopy category of Gorenstein-projective modules. As a byproduct, we give a sufficient and necessary condition on the upper triangular matrix algebra $T_{n}(A)$ to be of finite CM-type for an algebra $A$ of finite CM-type.
\end{abstract}

\maketitle

%\setprotcode\font
%    {\it \setprotcode\font}
%    {\bf \setprotcode\font}
%    {\bf \it \setprotcode\font}
%    \pdfprotrudechars=1

\vskip 20pt

\section{\bf Introduction}

\vskip 5pt

Derived categories, introduced by A.Grothendieck and J.L.Verdier \cite{V}, play an increasingly important role in various areas of mathematics, including representation theory, algebraic geometry, microlocal analysis and mathematical physics. Major topics of current interest include substructures of derived categories, such as bounded  $t$-structures, which form the ``skeleton" of Bridgeland's stability manifold, as well
as comparisons of derived categories.

\vskip 10pt

Recollements of triangulated categories were introduced by Beilinson, Bernstein and Deligne \cite{BBD} as a tool to get
information about the derived category of sheaves over a topological space $X$ from the corresponding derived categories
for an open subset $U\subseteq X$ and its complement $F=X\backslash U$. To any recollement one canonically associates the
TTF triples. Moreover, it is natural to look for conditions under which recollements of triangulated
categories at the ``bounded" level lift to recollements at the ``unbounded" level. Recently, M.Saor\'{\i}n and A.Zvonareva
\cite{SZ} gave a criterion for a recollement of triangulated subcategories to lift to a torsion torsion-free triple (TTF triple) of
ambient triangulated categories with coproducts.

\vskip 10pt

The study of Gorenstein homological algebra is due to Enochs and Jenda \cite{EJ1}. They introduced the concept of Gorenstein-projective modules, which are as a generalization of finitely generated modules of G-dimension zero over a two-sided Noetherian ring, in the sense of Auslander and Bridger \cite{AB}. The main idea of Gorenstein homological algebra is to replace projective modules by Gorenstein-projective modules, which is useful to study some Gorenstein properties. To intend to close a gap of the corresponding version of derived categories in
Gorenstein homological algebra, Gao and Zhang \cite{GZ} introduce Gorenstein derived category. Recently, Zhang \cite{Zh13} characterized the recollement of the stable categories of Gorenstein-projective modules over the triangular matrix algebra. Gao-Xu \cite{GX} showed that the recollement of derived categories of algebras induced by some homological ring epimorphism produces a ladder of the stable categories of Gorenstein-projective modules over corresponding algebras.

\vskip 10pt

In the paper, we investigate the lifting of recollements of bounded Gorenstein derived categories and the stable categories of Gorenstein-projective modules, respectively, where ladder is a crucial tool. We show that a homological ring epimorphism can induce a lifting of the recollement of the stable category of finitely generated Gorenstein-projective modules to non-finitely generated case (see Theorem~\ref{fg}). We prove that the recollement of the bounded Gorenstein derived categories of some upper triangular matrix algebra can be lifted to the homotopy category of Gorenstein-projective modules (see Theorem~\ref{Go} and Example~\ref{upper}). As a byproduct, we give a sufficient and necessary condition on the upper triangular matrix algebra $T_{n}(A)$ to be of finite CM-type for an algebra $A$ of finite CM-type (see Proposition~\ref{finiteCM}).

\vskip 10pt

\section{\bf Preliminaries}

\vskip 5pt

In this section, we fix notations and recall some basic concepts.

\vskip 10pt

Let $A$ be an Artin algebra. Denote by $A\mbox{-}{\rm Mod}$ (resp. $A\mbox{-}{\rm mod}$) the category of (resp. finitely generated) left $A\mbox{-}$modules, and $A\mbox{-}{\rm P}$ (resp. $A\mbox{-}{\rm proj}$) the full category of (resp. finitely generated) projective $A\mbox{-}$modules. A module $G$ of $A\mbox{-}{\rm Mod}$ (resp. $A\mbox{-}{\rm mod}$) is Gorenstein-projective if there is an exact sequence
$$\cdots \lxr P^{-1}\lxr P^{0}\stackrel{d^{0}}{\lxr} P^{1}\lxr \cdots$$
in $A\mbox{-}{\rm P}$ (resp. $A\mbox{-}{\rm proj}$), which stays exact after applying ${\rm Hom}_{A}(-, P)$ for each module $P$ in $A\mbox{-}{\rm P}$ (resp. $A\mbox{-}{\rm proj}$), such that $G\iso \Ker d^{0}$ (see \cite{EJ1, EJ2}). Denote by $A\mbox{-}{\rm GP}$ (resp. $A\mbox{-}{\rm Gproj}$) the full subcategories of Gorenstein-projective modules in $A\mbox{-}{\rm Mod}$ (resp. $A\mbox{-}{\rm mod}$), and $A\mbox{-}\underline{{\rm GP}}$ (resp. $A\mbox{-}\underline{{\rm Gproj}}$) the stable category of $A\mbox{-}{\rm GP}$ (resp. $A\mbox{-}{\rm Gproj}$) that modulo $A\mbox{-}{\rm P}$ (resp. $A\mbox{-}{\rm proj}$). Similarly, denote by $A\mbox{-}{\rm GI}$ the full subcategory of Gorenstein-injective modules in $A\mbox{-}{\rm Mod}$.

\vskip 10pt

Recall from \cite{B11} that an Artin algebra $A$ is of finite Cohen-Macaulay type (resp. finite CM-type for simply), if there are only finitely many isomorphism classes of indecomposable finitely generated Gorenstein-projective modules. Recall from \cite{H} that an Artin algebra $A$ is Gorenstein if ${\rm inj.dim} {_{A}A}<\infty$ and ${\rm inj.dim} A_{A}<\infty$. Recall from \cite{B05,BR} that an Artin algebra is called virtually Gorenstein if $(A\mbox{-}{\rm GP})^{\perp}= {^{\perp}(A\mbox{-}{\rm GI})}$, where
$$(A\mbox{-}{\rm GP})^{\perp}:=\{X\in A\mbox{-}{\rm Mod} \mid {\rm Ext}_{A}^{i}(G, X)=0, \ \forall i>0 \ and \ \forall G\in A\mbox{-}{\rm GP}\}$$
and
$$^{\perp}(A\mbox{-}{\rm GI}):=\{Y\in A\mbox{-}{\rm Mod}\mid {\rm Ext}_{A}^{i}(Y, I)=0, \ \forall i>0 \ and \ \forall I\in A\mbox{-}{\rm GI}\}.$$
Note that a Gorenstein algebra is virtually Gorenstein, but in general, the converse is not true.

\vskip 10pt

Now we write $C^{b}(A\mbox{-}{\rm Mod}),\ K^{b}(A\mbox{-}{\rm Mod})$ and $D^{b}(A\mbox{-}{\rm Mod})$ (resp. $C^{b}(A),\ K^{b}(A)$ and $D^{b}(A)$) for the bounded complex category, bounded homotopy category and bounded derived category of $A\mbox{-}{\rm Mod}$ (resp. $A\mbox{-}{\rm mod}$), respectively. Denote by $K(A\mbox{-}{\rm GP})$ (resp. $K(A\mbox{-}{\rm Gproj})$) the corresponding homotopy category of (finitely generated) Gorenstein-projective modules.

\vskip 10pt

A complex $C^{\bullet}$ of (finitely generated) $A\mbox{-}$modules is ${\rm GP}\mbox{-}$acyclic (resp. ${\rm Gproj}\mbox{-}$ acyclic), if ${\rm Hom}_{A}(G, C^{\bullet})$ is acyclic for each $G\in A\mbox{-}{\rm GP}$ (resp. $A\mbox{-}{\rm Gproj}$). A chain map $f^{\bullet}: X^{\bullet}\lxr Y^{\bullet}$ is a ${\rm GP}\mbox{-}$quasi-isomorphism (resp. ${\rm Gproj}\mbox{-}$quasi-isomorphism), if ${\rm Hom}_{A}(G, f^{\bullet})$ is a quasi-isomorphism for each $G\in A\mbox{-}{\rm GP}$ (resp. $A\mbox{-}{\rm Gproj}$), i.e., there are isomorphisms of abelian groups for any $n\in \mathbb{Z}$,
$$H^{n}{\rm Hom}_{A}(G, f^{\bullet}): H^{n}{\rm Hom}_{A}(G, X^{\bullet})\cong H^{n}{\rm Hom}_{A}(G, Y^{\bullet}).$$
Put
$$K^{b}_{gpac}(A\mbox{-}{\rm Mod}):=\{ X^{\bullet}\in K^{b}(A\mbox{-}{\rm Mod})\mid X^{\bullet} \ is \ {\rm GP}\mbox{-}acyclic\}$$
and
$$K^{b}_{gpac}(A):=\{ X^{\bullet}\in K^{b}(A)\mid X^{\bullet} \ is \ \Gproj\mbox{-}acyclic\}.$$
Then $K^{b}_{gpac}(A\mbox{-}{\rm Mod})$ (resp. $K^{b}_{gpac}(A)$) is a thick triangulated subcategory of $K^{b}(A\mbox{-}$ ${\rm Mod})$ (resp. $K^{b}(A)$). Following~\cite{GZ}, the Verdier quotient
$$D_{gp}^{b}(A\mbox{-}{\rm Mod}):=K^{b}(A\mbox{-}{\rm Mod})/ K^{b}_{gpac}(A\mbox{-}{\rm Mod})\ \ (resp.\ D_{gp}^{b}(A):=K^{b}(A)/ K^{b}_{gpac}(A)),$$
which is called the bounded Gorenstein derived category.

\vskip 15pt

Let $\D$ be a triangulated category. From \cite{IY}, a torsion pair in $\D$ is a pair of full subcategories $(\X, \Y)$, which are closed under direct summands, and satisfy ${\rm Hom}_{\D}(\X, \Y)=0$ and $\D=\X \ast \Y$, where
$$\X \ast \Y:= \{M\in \D \mid \exists\ a \ triangle\  X\lxr M\lxr Y\lxr X[1] \ with\ X\in \X,\ Y\in \Y \}.$$
A torsion pair is called a t-structure when $\X[1]\subseteq \X\ (\Leftrightarrow \Y[-1]\subseteq \Y)$ (see \cite{BBD}).
A TTF triple on $\D$ is a triple $(\U, \V, \W)$ of full subcategories of $\D$ such that $(\U, \V)$ and $(\V, \W)$ are t-structures on $\D$.

\vskip 5pt

Let $\D, \X$ and $\Y$ be triangulated categories. Recall from \cite{BBD} that $\D$ is said to be a recollement of $\X$ and $\Y$ if there are six triangle functors as in the following diagram
\[
\xymatrix@C=0.5cm{
\Y \ar[rrr]^{i_{\ast}} &&&
\D \ar[rrr]^{j^{\ast}}  \ar @/_1.5pc/[lll]_{i^{\ast}}  \ar @/^1.5pc/[lll]^{i^{!}} &&&
\X \ar @/_1.5pc/[lll]_{j_{!}}  \ar @/^1.5pc/[lll]^{j_{\ast}}
 }
\]
such that
\begin{enumerate}

\item \ $(i^{\ast}, i_{\ast}),\ (i_{\ast}, i^{!}),\ (j_{!}, j^{\ast})$ and $(j^{\ast}, j_{\ast})$ are adjoint pairs;

\vskip 5pt

\item \ $i_{\ast},\ j_{!}$ and $j_{\ast}$ are fully faithful;

\vskip 5pt

\item \ $j^{\ast}i_{\ast}=0$;

\vskip 5pt

\item \ For each $Z\in \D$, the counits and units give rise to distinguished triangles:
$$j_{!}j^{\ast}Z\lxr Z\lxr i_{\ast}i^{\ast}Z\lxr \ \ \ and\ \ \ i_{\ast}i^{!}Z\lxr Z\lxr j_{\ast}j^{\ast}Z\lxr.$$
\end{enumerate}
We simply denote the recollement by $(\Y\equiv \D\equiv \X)$. It is well known that TTF triples are in bijection with (equivalence classes of) recollement data \cite{BBD,Ne,Ni}. If there is a recollement as above, then $(\im j_{!}, \im i_{\ast}, \im j_{\ast})$ is a TTF triple in $\D$. Conversely, if $(\X, \Y, \Z)$ is a TTF triple in $\D$, then one obtains a recollement as above, where $j_{!}: \X\hookrightarrow \D$ and $i_{\ast}: \Y\hookrightarrow \D$ are the inclusion functors.

\vskip 10pt

Two recollements
\[
\xymatrix@C=0.4cm{
\Y \ar[rrr]^{i_{\ast}} &&&
\D \ar[rrr]^{j^{\ast}}  \ar @/_1.5pc/[lll]_{i^{\ast}}  \ar @/^1.5pc/[lll]^{i^{!}} &&&
\X \ar @/_1.5pc/[lll]_{j_{!}}  \ar @/^1.5pc/[lll]^{j_{\ast}}
}\ \ \ \ \ \
\xymatrix@C=0.4cm{
\Y' \ar[rrr]^{i'_{\ast}} &&&
\D' \ar[rrr]^{j'^{\ast}}  \ar @/_1.5pc/[lll]_{i'^{\ast}}  \ar @/^1.5pc/[lll]^{i'^{!}} &&&
\X' \ar @/_1.5pc/[lll]_{j'_{!}}  \ar @/^1.5pc/[lll]^{j'_{\ast}}
}
\]
are said to be equivalent, if $(\im j_{!}, \im i_{\ast}, \im j_{\ast})=(\im j'_{!}, \im i'_{\ast}, \im j'_{\ast})$.

\vskip 10pt

A ladder of recollements $\mathcal{L}$ is a finite or infinite diagram of triangulated categories and triangle functors
\[
\xymatrix@C=3em@R=3.5em{
& \vdots &  & \vdots & \\
\D'   \ar[rr]^{i_{0}}   \ar@/^3pc/[rr]^{i_{2}}   \ar@/_3pc/[rr]^{i_{-2}}   &  &
\D   \ar[rr]^{j_{0}}    \ar@/_1.5pc/[ll]_{j_{1}}     \ar@/^1.5pc/[ll]_{j_{-1}} \ar@/^3pc/[rr]^{j_{2}}  \ar@/_3pc/[rr]^{j_{-2}} & &
\D''    \ar@/_1.5pc/[ll]_{i_{1}}   \ar@/^1.5pc/[ll]_{i_{-1}}\\
& \vdots &  & \vdots &\\
 }
\]
such that any consecutive rows form a recollement (see \cite[Section 3]{AKLY}; \cite[Section 1.5]{BGS}). The height of a ladder is the number of recollements contained in it (counted with multiplicities).

\vskip 10pt

Let $\D$ be a triangulated category with coproducts. A triangulated subcategory closed under arbitrary coproducts is called a localizing subcategory. Given a class $\cS$ of objects of $\D$, denote by $\Loc_{\D}(S)$ the smallest localizing subcategory containing $\cS$. A compact object is an object $M\in \D$ such that $\coprod_{i\in I}{\rm Hom}_{\D}$ $ (M, N_{i})\iso {\rm Hom}_{\D}(M, \coprod_{i\in I}N_{i})$, for each family of objects $(N_{i})_{i\in I}$. We say that $\D$ is compactly generated  when it has a generating set of compact objects. Denote by $\D^{c}$ the subcategory of compact objects of $\D$.

\vskip 5pt

For a class of objects $\cS$ in $\D$, we consider the following subcategories of $\D$:
$$\cS^{\perp}=\{M\in \D\mid {\rm Hom}_{\D}(S, M)=0\ for \ any \ S\in \cS\}$$
and
$$^{\perp}\cS=\{M\in \D\mid {\rm Hom}_{\D}(M, S)=0\ for \ any \ S\in \cS\}.$$
We say that $\cS$ generates a torsion pair in $\D$ if $({^{\perp}(\cS^{\perp})}, \cS^{\perp})$ is a torsion pair.
If moreover $\cS$ is a set of compact objects, $({^{\perp}(\cS^{\perp})}, \cS^{\perp})$ is called compactly generated.

\vskip 10pt

\section{\bf Lifting of TTF triples}

\vskip 10pt

In this section, we investigate the lifting of recollements of bounded Gorenstein derived categories and the stable categories of Gorenstein-projective modules, respectively, where ladder is a crucial tool. Moreover, we give a sufficient and necessary condition on the upper triangular matrix algebra $T_{n}(A)$ to be of finite CM-type for an algebra $A$ of finite CM-type.

\vskip 10pt

\begin{defn}(\cite[Definition 2]{NiS}) \ Let $\D$ be a triangulated category and $\D'$ the full triangulated subcategory of $\D$. We say that a TTF triple $(\U, \V, \W)$ in $\D$ restrict to or is a lifting of a TTF triple $(\U', \V', \W')$ in $\D'$, \  if we have
$$(\U\cap \D', \V'\cap \D', \W'\cap \D')=(\U', \V', \W').$$
\end{defn}

\vskip 10pt

Recall that, given thick subcategories $\Y'\subseteq \Y,\ \D'\subseteq \D$ and $\X'\subseteq \X$ and recollements
\[
\xymatrix@C=0.4cm{
\Y \ar[rrr]^{i_{\ast}} &&&
\D \ar[rrr]^{j^{\ast}}  \ar @/_1.5pc/[lll]_{i^{\ast}}  \ar @/^1.5pc/[lll]^{i^{!}} &&&
\X \ar @/_1.5pc/[lll]_{j_{!}}  \ar @/^1.5pc/[lll]^{j_{\ast}}
}\ \ \ \ \ \
\xymatrix@C=0.4cm{
\Y' \ar[rrr]^{i'_{\ast}} &&&
\D' \ar[rrr]^{j'^{\ast}}  \ar @/_1.5pc/[lll]_{i'^{\ast}}  \ar @/^1.5pc/[lll]^{i'^{!}} &&&
\X' \ar @/_1.5pc/[lll]_{j'_{!}}  \ar @/^1.5pc/[lll]^{j'_{\ast}}
}
\]
We say that the recollement $(\Y\equiv \D\equiv \X)$ restricts, up to equivalence, to the recollement $(\Y'\equiv \D'\equiv \X')$, or that the recollement $(\Y'\equiv \D'\equiv \X')$ lifts, up to equivalence, to the recollement $(\Y\equiv \D\equiv \X)$, when $(\im j_{!}\cap \D', \im i_{\ast}\cap \D', \im j_{\ast}\cap \D')=(\im j'_{!}, \im i'_{\ast}, \im j'_{\ast})$.

\vskip 10pt

Next we are interest in compactly generated triangulated categories. The lemma we give below is implicit in \cite[Theorem 3.3]{SZ}.
In the following, when we say TTF quadruple $(\U, \V, \W, \Z)$, we mean that both $(\U, \V, \W)$ and $(\V, \W, \Z)$ are TTF triples. TTF tuple can be similarly defined.

\vskip 10pt

\begin{lem}\
\label{TTFlift}\
Let
\[
\xymatrix@C=0.5cm{
\Y \ar[rrr]^{i_{\ast}} \ar @/_3.0pc/[rrr]^{i}  &&&
\D \ar[rrr]^{j^{\ast}}  \ar @/_3.0pc/[rrr]^{j}  \ar @/_1.5pc/[lll]_{i^{\ast}}  \ar @/^1.5pc/[lll]_{i^{!}} &&&
\X \ar @/_1.5pc/[lll]_{j_{!}}  \ar @/^1.5pc/[lll]_{j_{\ast}}
 }
\]
be a ladder of recollements of height 2, and $\Y, \D, \X$ the thick subcategories of compactly generated triangulated category $\hat{\Y}, \hat{\D}, \hat{\X}$ which contain the respective subcategories of compact objects.
\begin{enumerate}

\item \ The given recollement can lift to a ladder of recollements of height 3:
\[
\xymatrix@C=0.5cm{
\hat{\Y} \ar[rrr]^{\hat{i_{\ast}}} \ar @/_3.0pc/[rrr]^{\hat{i}}  &&&
\hat{\D} \ar[rrr]^{\hat{j^{\ast}}}  \ar @/_3.0pc/[rrr]^{\hat{j}}  \ar @/_1.5pc/[lll]_{\hat{i^{\ast}}}  \ar @/^1.5pc/[lll]_{\hat{i^{!}}} \ar @/^4.5pc/[lll]_{} &&&
\hat{\X}. \ar @/_1.5pc/[lll]_{\hat{j_{!}}}  \ar @/^1.5pc/[lll]_{\hat{j_{\ast}}} \ar @/^4.5pc/[lll]_{}
 }
\]

\vskip 5pt

\item \ The functors $i^{!}$ and $j_{\ast}$ preserve compact objects, i.e. $i^{!}(\hat{\D}^{c})\subseteq \hat{\Y}^{c}$ and $j_{\ast}(\hat{\X}^{c})\subseteq \hat{\D}^{c}$. Moreover, $\im i_{\ast}$ cogenerates $\Loc_{\hat{\D}}(i_{\ast}(\hat{\Y}^{c}))$ and $\im j_{\ast}$ cogenerates $\Loc_{\hat{\D}}(j_{\ast}(\hat{\X}^{c}))$, or $\D$ cogenerates $\hat{\D}$.
\end{enumerate}
If one of the above two conditions is satisfied, then the TTF quadruple $(\im j_{!}, \im i_{\ast},$ $\im j_{\ast}, \im i)$ lifts to a TTF quadruple $(\U, \V, \W, \Z)$ in $\hat{\D}$ such that:
\vskip 5pt
{\rm (a):} $(\U, \V), (\V, \W)$ and $(\W, \Z)$ are compactly generated;
\vskip 5pt
{\rm (b):} $j_{!}(\hat{\X}^c)= \U\cap \hat{\D}^{c},\ i_{\ast}(\hat{\Y}^{c})= \V\cap \hat{\D}^{c}$ and $j_{\ast}(\hat{\X}^{c})= \W\cap \hat{\D}^{c}$.
\end{lem}
\begin{proof}\ It can be seen from \cite[Theorem 3.3]{SZ} that when condition (1) or (2) is satisfied, there are two lifted TTF triples $(\U, \V, \W)$ and $(\V, \W, \Z)$ in $\hat{\D}$, such that (a) and (b) hold.
\end{proof}

\vskip 10pt

\begin{thm}\
\label{fg}
Let $A$ be a Gorenstein algebra, and $\lambda: A\lxr B$ a homological ring epimorphism which induces a recollement of derived categories \[
\xymatrix@C=0.4cm{
D(B\mbox{-}{\rm Mod}) \ar[rrr]^{i_{\ast}} &&&
D(A\mbox{-}{\rm Mod}) \ar[rrr]^{j^{\ast}}  \ar @/_1.5pc/[lll]_{i^{\ast}}  \ar @/^1.5pc/[lll]_{i^{!}} &&&
D(C\mbox{-}{\rm Mod}) \ar @/_1.5pc/[lll]_{j_{!}}  \ar @/^1.5pc/[lll]_{j_{\ast}}
}
\]
\vskip 5pt
\noindent \ of algebras $B,\ A$ and $C$ such that $j_{!}$ restricts to $D^{b}(C\mbox{-}{\rm mod})$. Assume that ${\rm pd}{_{A}B}<\infty$, then the TTF tuple in $A\mbox{-}\underline{{\rm Gproj}}$ lifts to a TTF tuple in $A\mbox{-}\underline{{\rm GP}}$.
\end{thm}
\begin{proof}\ Following \cite[Theorem 3.1]{GX}, there is an unbounded ladder
\[
\xymatrix@C=3em@R=3.5em{
& \vdots &  & \vdots & \\
B\mbox{-}\underline{{\rm Gproj}}   \ar[rr]^{i_{\ast}}   \ar@/^3pc/[rr]^{i_{?}}   \ar@/_3pc/[rr]^{i_{@}}   &  &
A\mbox{-}\underline{{\rm Gproj}}   \ar[rr]^{j^{\ast}}    \ar@/_1.5pc/[ll]_{i^{\ast}}     \ar@/^1.5pc/[ll]_{i^{!}} \ar@/^3pc/[rr]^{j^{?}}  \ar@/_3pc/[rr]^{j^{@}} & &
C\mbox{-}\underline{{\rm Gproj}}    \ar@/_1.5pc/[ll]_{j_{!}}   \ar@/^1.5pc/[ll]_{j_{\ast}}\\
& \vdots &  & \vdots &\\
 }
\]
Since the functors $i_{@}$ and $j^{@}$ in this ladder have right adjoints, they preserve coproducts, and  so the functors $i^{!}$ and $j_{\ast}$ preserve compact objects. Furthermore, it follows from \cite[Theorem 4.1]{Ch11} that $A\mbox{-}\underline{{\rm GP}}$ is a compactly generated triangulated category, and $(A\mbox{-}\underline{{\rm GP}})^{c}=A\mbox{-}\underline{{\rm Gproj}}$. For
$$(A\mbox{-}\underline{{\rm GP}})^{b}:=\{M\mid {\rm Hom}_{A\mbox{-}\underline{{\rm GP}}}(X, M[k])=0 \ for \ each\ X\in A\mbox{-}\underline{{\rm Gproj}} \ and \ |k|\gg 0\},$$
we have that $A\mbox{-}\underline{{\rm Gproj}}\subseteq (A\mbox{-}\underline{{\rm GP}})^{b}$. This means  from \cite[Corollary 3.12]{SZ} that $A\mbox{-}\underline{{\rm Gproj}}$ cogenerates $A\mbox{-}\underline{{\rm GP}}$. By Lemma~\ref{TTFlift}, we have that the TTF quadruple $(\im j_{!}, \im i_{\ast},$ $\im j_{\ast}, \im i_{@})$ lifts to a TTF quadruple $(\U, \V, \W, \Z)$ in $A\mbox{-}\underline{{\rm GP}}$ such that (a) and (b) hold. Consequently, we get the lifting of TTF tuple in the unbounded ladder.
\end{proof}

\vskip 10pt

The lemma given below will be very useful.

\vskip 10pt

\begin{lem}{\rm (\cite[Corollary 3.5]{SZ})}\
\label{compact}
Let $\Y,\ \D$ and $\X$ be compactly generated triangulated categories and suppose that we have a recollement
\[
\xymatrix@C=0.4cm{
\Y^{c} \ar[rrr]^{i_{\ast}} &&&
\D^{c} \ar[rrr]^{j^{\ast}}  \ar @/_1.5pc/[lll]_{i^{\ast}}  \ar @/^1.5pc/[lll]^{i^{!}} &&&
\X^{c}. \ar @/_1.5pc/[lll]_{j_{!}}  \ar @/^1.5pc/[lll]^{j_{\ast}}
}
\]
Then the TTF triple $(\im j_{!}, \im i_{\ast}, \im j_{\ast})$ in $\D^{c}$ lifts to a TTF triple $(\U, \V, \W)$ in $\D$, where the torsion pairs $(\U, \V)$ and $(\V, \W)$ are compactly generated.
\end{lem}

\vskip 10pt

Immediately, we can draw the following conclusions.

\vskip 5pt

\begin{cor}\ Let $\Lambda=\begin{pmatrix}
A & N \\
0 & B \\
\end{pmatrix}$ be a Gorenstein Artin algebra such that $M$ is a projective $A$-module. Then the TTF triple in $\La\mbox{-}\underline{{\rm Gproj}}$ can lift to a TTF triple in $\La\mbox{-}\underline{{\rm GP}}$.
\end{cor}
\begin{proof}\  On one hand, since $\La$ is Gorenstein and  $M$ is $A$-projective, we have from \cite[Theorem 2.2]{Zh13} that $A$ and $B$ are Gorenstein. Furthermore, it follows from \cite[Theorem 4.1]{Ch11} that $\La\mbox{-}\underline{{\rm GP}}$,  $A\mbox{-}\underline{{\rm GP}}$ and $B\mbox{-}\underline{{\rm GP}}$ are  compactly generated triangulated categories, and moreover, we have the following equalities:
$$(\La\mbox{-}\underline{{\rm GP}})^{c}=\La\mbox{-}\underline{{\rm Gproj}}, \ (A\mbox{-}\underline{{\rm GP}})^{c}=A\mbox{-}\underline{{\rm Gproj}},\ (B\mbox{-}\underline{{\rm GP}})^{c}=B\mbox{-}\underline{{\rm Gproj}}.$$

\vskip 10pt

On the other hand, there exists a recollement by \cite[Theorem 3.5]{Zh13} as follows:
\[
\xymatrix@C=0.5cm{
\ A\mbox{-}\underline{{\rm Gproj}} \ar[rrr]^{i_{\ast}} &&&
\La\mbox{-}\underline{{\rm Gproj}} \ar[rrr]^{j^{\ast}}  \ar @/_1.5pc/[lll]_{i^{\ast}}  \ar @/^1.5pc/[lll]^{i^{!}} &&&
B\mbox{-}\underline{{\rm Gproj}}. \ar @/_1.5pc/[lll]_{j_{!}}  \ar @/^1.5pc/[lll]^{j_{\ast}}
 }
\]
Thus we get from Lemma~\ref{compact} that the TTF triple $(\im j_{!}, \im i_{\ast}, \im j_{\ast})$ in $\La\mbox{-}\underline{{\rm Gproj}}$ lifts to a TTF triple $(\U, \V, \W)$ in $\La\mbox{-}\underline{{\rm GP}}$, and that the torsion pairs $(\U, \V)$ and $(\V, \W)$ are compactly generated.
\end{proof}

\vskip 10pt

\begin{lem}\
\label{finiteCM}\
Let $A, B$ and $C$ be virtually Gorenstein algebras of finite CM-type. Assume that $D_{gp}^{b}(A)$ admits the following recollement:
\[
\xymatrix@C=0.5cm{
\ D_{gp}^{b}(B) \ar[rrr]^{i_{\ast}} &&&
D_{gp}^{b}(A) \ar[rrr]^{j^{\ast}}  \ar @/_1.5pc/[lll]_{i^{\ast}}  \ar @/^1.5pc/[lll]^{i^{!}} &&&
D_{gp}^{b}(C). \ar @/_1.5pc/[lll]_{j_{!}}  \ar @/^1.5pc/[lll]^{j_{\ast}}
 }
\]
Then the TTF triple $(\im j_{!}, \im i_{\ast}, \im j_{\ast})$ in $D_{gp}^{b}(A)$ lifts to a TTF triple $(\U, \V, \W)$ in $K(A\mbox{-}{\rm GP})$.
\end{lem}
\begin{proof}\ Since $A, B$ and $C$ are virtually Gorenstein algebras of finite CM-type, we get from \cite[Theorem 3.2]{G15} that three homotopy categories $K(A\mbox{-}{\rm GP}), \ K(B\mbox{-}{\rm GP})$ and $K(C\mbox{-}{\rm GP})$ are compactly generated, and moreover,
$$(K(A\mbox{-}{\rm GP}))^{c}\cong D_{gp}^{b}(A), \ \ \ (K(B\mbox{-}{\rm GP}))^{c}\cong D_{gp}^{b}(B)$$
and
$$(K(C\mbox{-}{\rm GP}))^{c}\cong D_{gp}^{b}(C).$$
Thus the result follows from Lemma~\ref{compact}.
\end{proof}

\vskip 10pt

\begin{exam}\
\label{upper}
Let $A$ be a Gorenstein algebra of finite CM-type, and $T_{n}(A)$ the upper triangular matrix algebra of $A$. Then the following hold.
\begin{enumerate}
\item \ $D_{gp}^{b}(T_{n}(A))$ admits a recollement:
\[
\xymatrix@C=0.5cm{
\ D_{gp}^{b}(T_{n-1}(A)) \ar[rrr]^{i_{\ast}} &&&
D_{gp}^{b}(T_{n}(A)) \ar[rrr]^{j^{\ast}}  \ar @/_1.5pc/[lll]_{i^{\ast}}  \ar @/^1.5pc/[lll]^{i^{!}} &&&
D_{gp}^{b}(A). \ar @/_1.5pc/[lll]_{j_{!}}  \ar @/^1.5pc/[lll]^{j_{\ast}}
 }
\]

\vskip 5pt

\item \  The TTF triple $(\im j_{!}, \im i_{\ast}, \im j_{\ast})$ in $D_{gp}^{b}(T_{n}(A))$ lifts to a TTF triple $(\U, \V, \W)$ in $K(T_{n}(A)\mbox{-}{\rm GP})$.
\end{enumerate}
\end{exam}
\begin{proof}\ (1)\ Since $T_{n}(A)$ admits a ladder of recollements of l-height 2 and r-height 4 as follows:
\[
\xymatrix@C=0.5cm{
T_{n-1}(R)\mbox{-}{\rm mod} \ar[rrr]^{\mi} &&&
T_{n}(R)\mbox{-}{\rm mod} \ar @/_3.0pc/[rrr]^{\mr^1} \ar[rrr]^{\me} \ar @/^3.0pc/[rrr]^{\ml^1} \ar @/_1.5pc/[lll]_{\mq}  \ar @/^1.5pc/[lll]^{\map} \ar @/_4.5pc/[rrr]^{\mr^3} &&&
R\mbox{-}{\rm mod}, \ar @/^4.5pc/[lll]_{\mr^2} \ar @/_1.5pc/[lll]_{\ml} \ar @/^1.5pc/[lll]_{\mr}
 }
\]
such that $\mq$ and $\mi$ are exact. It follows from \cite[Theorem 7.4]{GKP} that $\ml,\ \me$ and $\mr$ preserve Gorenstein-projective modules. Moreover, by definitions of $\mq,\ \mi$ and $\map$, we have that they preserve Gorenstein projective modules. It induces the following recollement of Gorenstein derived categories:
\[
\xymatrix@C=0.5cm{
D_{gp}^{b}(T_{n-1}(A)) \ar[rrr]^{D^{b}(\mi)} &&&
D_{gp}^{b}(T_{n}(A)) \ar[rrr]^{D^{b}(\me)}  \ar @/_1.5pc/[lll]_{D^{b}(\mq)}  \ar @/^1.5pc/[lll]^{D^{b}(\map)} &&&
D_{gp}^{b}(A). \ar @/_1.5pc/[lll]_{D^{b}(\ml)}  \ar @/^1.5pc/[lll]^{D^{b}(\mr)}
 }
\]

\vskip 5pt

(2)\ Since $A$ is Gorenstein, it follows that $T_{n}(A)$ and $T_{n-1}(A)$ are Gorenstein. Since $A$ is CM-finite, it follows from \cite[Example 3.2]{EHSL} that $T_{n}(A)$ and $T_{n-1}(A)$ are CM-finite. Thus we can get from Lemma~\ref{finiteCM} that the TTF triple in $D_{gp}^{b}(T_{n}(A))$ can lift to a TTF triple in $K(T_{n}(A)\mbox{-}{\rm GP})$.
\end{proof}

\vskip 10pt

In fact, we will give a criterion for the upper triangulated matrix algebra $T_{n}(A)$ to be of finite CM-type for the Gorenstein algebra $A$. Before this, we need to make some preparations.

\vskip 10pt

\noindent{\bf Construction.}\ \ By an $(A\mbox{-}{\rm Gproj})\mbox{-}$module, we mean an additive contravariant functor $F: A\mbox{-}{\rm Gproj}\lxr \Ab$, where $\Ab$ denotes the category of abelian groups.

\vskip 10pt

We say that $X$ lies in the homotopy equivalence classes of $Y$, which means that the projective resolutions of $X$ and $Y$ are homotopy equivalent. We denote by $\cht(A\mbox{-}{\rm mod})$ the category of the homotopy equivalence classes of $A$-modules. Put
$${\pGp}^{\leq n}(A):=\{X\in \cht(A\mbox{-}{\rm mod})\mid \exists \ a \ proper \ Gorenstein\mbox{-}\ projective \ resolution$$
$$0\to G_{n}\to G_{n-1}\to \cdots \to G_{0}\to X\to 0 \}$$
and
$${\pd}^{\leq n}((A\mbox{-}{\rm Gproj})\mbox{-}{\rm mod}):=\{F\in \cht((A\mbox{-}{\rm Gproj})\mbox{-}{\rm mod})\mid \exists \ a \ projective \ resolution$$
$$0\to (-, G_{n})\to (-, G_{n-1})\to \cdots \to (-, G_{0})\to F\to 0 \}.$$

\vskip 10pt

Let $X\in {\pGp}^{\leq n}(A)$. Consider the proper Gorenstein-projective resolution of $X$:
$$0\lxr G_{n}\stackrel{d_{n}}{\lxr} G_{n-1}\lxr \cdots \lxr G_{1}\stackrel{d_{1}}{\lxr} G_{0}\lxr X\lxr 0.$$
Then we  get the following exact sequence:
$$0\lxr (-, G_{n})\stackrel{(-, d_{n})}{\lxr} (-, G_{n-1})\lxr \cdots \lxr (-, G_{1})\stackrel{(-, d_{1})}{\lxr} (-, G_{0})\lxr F_{X}\lxr 0.$$
This means that $F_{X}\in {\pd}^{\leq n}((A\mbox{-}{\rm Gproj})\mbox{-}{\rm mod})$. Now we define
$$\F: {\pGp}^{\leq n}(A)\lxr {\pd}^{\leq n}((A\mbox{-}{\rm Gproj})\mbox{-}{\rm mod})$$
given by $\F (X)=F_{X}$. For a morphism $\alpha: X\lxr X^{\prime}$ in ${\pGp}^{\leq n}(A)$, there is the following commutative diagram
\[
\xymatrix{
0\ar[r]^{} & G_{n}\ar[r]^{d_{n}}\ar[d]_{\alpha_{n}} & G_{n-1}\ar[r]^{d_{n-1}}\ar[d]_{\alpha_{n-1}} & \cdots\ar[r] & G_{1}\ar[r]^{d_{1}}\ar[d]_{\alpha_{1}} & G_{0}\ar[r]^{}\ar[d]_{\alpha_{0}} & X\ar[r]^{}\ar[d]_{\alpha} & 0 \\
0\ar[r]^{} & G_{n}^{\prime}\ar[r]^{d_{n}^{\prime}} & G_{n-1}^{\prime}\ar[r]^{d_{n-1}^{\prime}} & \cdots\ar[r] & G_{1}^{\prime}\ar[r]^{d_{1}^{\prime}} & G_{0}^{\prime}\ar[r]^{} & X^{\prime}\ar[r]^{} & 0. }
\]
We define $\F (\alpha): F_{X}\lxr F_{X^{\prime}}$ to be the unique morphism $\overline{(-, \alpha_{0})}$, which is obtained by making the following diagram commutative
\[
\xymatrix{
0\ar[r]^{} & (-, G_{n})\ar[r]^{(-, d_{n})}\ar[d]_{(-, \alpha_{n})} & \cdots\ar[r] & (-, G_{1})\ar[r]^{(-, d_{1})}\ar[d]_{(-, \alpha_{1})} & (-, G_{0})\ar[r]^{}\ar[d]_{(-, \alpha_{0})} & F_{X}\ar[r]^{}\ar[d]_{\overline{(-, \alpha_{0})}} & 0 \\
0\ar[r]^{} & (-, G_{n}^{\prime})\ar[r]^{(-, d_{n}^{\prime})} & \cdots\ar[r] & (-, G_{1}^{\prime})\ar[r]^{(-, d_{1}^{\prime})} & (-, G_{0}^{\prime})\ar[r]^{} & F_{X^{\prime}}\ar[r]^{} & 0. }
\]
The composition of morphisms will be defined naturally. Clearly, $\F$ is an additive functor.

\vskip 10pt

Denote by $\Mor_{n}(A)$ the morphism category of $A$. It is well-known that there is an equivalence between $T_{n}(A)\mbox{-}{\rm mod}$ and $\Mor_{n}(A)$. Denote by ${\CS}_{n}(A)$ the full subcategory of $\Mor_{n}(A)$ consisting of all monomorphisms.

\vskip 10pt

\begin{lem}\ {\rm (\cite[Corollary 4.1(ii)]{Zh11})} \ Let $A$ be a Gorenstein algebra. Then
$$T_{n}(A)\mbox{-}{\rm Gproj}={\CS}_{n}(A\mbox{-}{\rm Gproj}).$$
\end{lem}

\vskip 10pt

We are ready to give the criterion for $T_{n}(A)$ to be of finite CM-type.

\vskip 5pt

\begin{prop}\
\label{finiteCM}
Let $A$ be an Artin algebra. Then there are equivalence of additive categories for any positive integer $n$
$$\F: {\pGp}^{\leq n}(A)\lxr {\pd}^{\leq n}((A\mbox{-}{\rm Gproj})\mbox{-}{\rm mod}).$$
Consequently, if $A$ is Gorenstein, then $T_{n}(A)$ is of finite CM-type if and only if ${\pd}^{\leq n}((A\mbox{-}{\rm Gproj})\mbox{-}{\rm mod})$ is of finite type.
\end{prop}
\begin{proof}\ We first show that $\F$ is full. Let $X$ and $X^{\prime}$ be in ${\pGp}^{\leq n}(A)$ and $\bar{\alpha}: \F(X)\lxr \F(X^{\prime})$ the morphism in ${\pd}^{\leq n}((A\mbox{-}{\rm Gproj})\mbox{-}{\rm mod})$. Obviously, this morphism can be lifted to the morphism between the projective resolutions of $\F(X)$ and $\F(X^{\prime})$, which are induced by the proper Gorenstein-projective resolutions of $X$ and $X^{\prime}$ respectively. Then by Yoneda's lemma there exists the morphism $\alpha: X\lxr X^{\prime}$ such that $\F(\alpha)=\bar{\alpha}$.

\vskip 5pt

Next we show that $\F$ is faithful. Let $\alpha: X\lxr X^{\prime}$ be a morphism in ${\pGp}^{\leq n}(A)$ such that $\F(\alpha)=0$. This implies that the morphism $\bar{\alpha}: F_{X}\lxr F_{X^{\prime}}$ is null-homotopic. From the following commutative diagram
\[
\xymatrix{
0\ar[r]^{} & (-, G_{n})\ar[r]^{(-, d_{n})}\ar[d]_{(-, \alpha_{n})} & \cdots\ar[r] & (-, G_{1})\ar[r]^{(-, d_{1})}\ar[d]_{(-, \alpha_{1})} & (-, G_{0})\ar[r]^{}\ar[d]_{(-, \alpha_{0})} & F_{X}\ar[r]^{}\ar[d]_{\bar{\alpha}} & 0 \\
0\ar[r]^{} & (-, G_{n}^{\prime})\ar[r]^{(-, d_{n}^{\prime})} & \cdots\ar[r] & (-, G_{1}^{\prime})\ar[r]^{(-, d_{1}^{\prime})} & (-, G_{0}^{\prime})\ar[r]^{} & F_{X^{\prime}}\ar[r]^{} & 0. }
\]
and Yoneda's lemma, there exists a series of morphisms $h_{i}: G_{i}\lxr G_{i+1}^{\prime}$ for $i= 0, 1, \cdots, n-1$ such that $\alpha_{0}=d_{1}^{\prime} h_{0}, \alpha_{i}=h_{i-1}d_{i}+d_{i+1}^{\prime}h_{i}$ and $\alpha_{n}=h_{n-1}d_{n}$ for $i= 1, 2, \cdots, n-1$, as shown in the diagram below:
\[
\xymatrix{
0\ar[r]^{} & G_{n}\ar[r]^{d_{n}}\ar[d]^{\alpha_{n}} & G_{n-1}\ar[r]^{d_{n-1}}\ar[d]^{\alpha_{n-1}}\ar[ld]^{h_{n-1}} & \cdots\ar[r] & G_{2}\ar[r]^{d_{2}}\ar[d]^{\alpha_{2}} & G_{1}\ar[r]^{d_{1}}\ar[d]^{\alpha_{1}}\ar[ld]^{h_{1}} & G_{0}\ar[d]^{\alpha_{0}}\ar[ld]^{h_{0}}\\
0\ar[r]^{} & G_{n}^{\prime}\ar[r]_{d_{n}^{\prime}} & G_{n-1}^{\prime}\ar[r]_{d_{n-1}^{\prime}} & \cdots\ar[r] & G_{2}^{\prime}\ar[r]_{d_{2}^{\prime}} & G_{1}^{\prime}\ar[r]_{d_{1}^{\prime}} & G_{0}^{\prime}.  }
\]
This implies that $\alpha$ is null-homotopic, and hence $\F$ is faithful.

\vskip 5pt

Finally, we show that $\F$ is dense. Let $F\in {\pd}^{\leq n}((A\mbox{-}{\rm Gproj})\mbox{-}{\rm mod})$. Taking a projective resolution of $F$ as follows
$$0\lxr (-, G_{m})\lxr (-, G_{m-1})\lxr \cdots \lxr (-, G_{1})\lxr (-, G_{0})\lxr F\lxr 0$$
with $m\leq n$, we get the exact sequence
$$0\lxr G_{m}\lxr G_{m-1}\lxr \cdots \lxr G_{1}\xrightarrow{d_{0}} G_{0}$$
which is the proper Gorenstein-projective resolution of ${\rm Coker}d_{0}$, and $\F({\rm Coker}d_{0})= F$.
\end{proof}

\vskip 10pt

\begin{thm}\
\label{Go}
Let $A, B$ and $C$ be virtually Gorenstein algebras of finite CM-type, and
\[
\xymatrix@C=0.5cm{
\ K(B\mbox{-}{\rm GP}) \ar[rrr]^{i_{\ast}} \ar @/_3.0pc/[rrr]^{}  &&&
K(A\mbox{-}{\rm GP}) \ar[rrr]^{j^{\ast}}  \ar @/_3.0pc/[rrr]^{}  \ar @/_1.5pc/[lll]_{i^{\ast}}  \ar @/^1.5pc/[lll]_{i^{!}} &&&
K(C\mbox{-}{\rm GP}) \ar @/_1.5pc/[lll]_{j_{!}}  \ar @/^1.5pc/[lll]_{j_{\ast}}
 }
\]
a ladder of recollements of height two. The following are equivalent:
\begin{enumerate}

\item \ The recollement can restrict to a recollement
\[
\xymatrix@C=0.5cm{
\ D_{gp}^{b}(B) \ar[rrr]^{i_{\ast}} &&&
D_{gp}^{b}(A) \ar[rrr]^{j^{\ast}}  \ar @/_1.5pc/[lll]_{i^{\ast}}  \ar @/^1.5pc/[lll]^{i^{!}} &&&
D_{gp}^{b}(C); \ar @/_1.5pc/[lll]_{j_{!}}  \ar @/^1.5pc/[lll]^{j_{\ast}}
 }
\]

\vskip 5pt

\item \ The associated TTF triple $(\im j_{!}, \im i_{\ast}, \im j_{\ast})$ can restrict to $D_{gp}^{b}(A)$.
\end{enumerate}

\vskip 5pt

If either of the above conditions holds, then $A$ is Gorenstein if and only if $B$ and $C$ are also.
\end{thm}
\begin{proof}\ Since $A, B$ and $C$ are virtually Gorenstein algebras of finite CM-type, we get from \cite[Theorem 3.2]{G15} that three homotopy categories $K(A\mbox{-}{\rm GP}), \ K(B\mbox{-}{\rm GP})$ and $K(C\mbox{-}{\rm GP})$ are compactly generated, and moreover,
$$(K(A\mbox{-}{\rm GP}))^{c}\cong D_{gp}^{b}(A), \ \ \ (K(B\mbox{-}{\rm GP}))^{c}\cong D_{gp}^{b}(B)$$
and
$$(K(C\mbox{-}{\rm GP}))^{c}\cong D_{gp}^{b}(C).$$
This means the equivalence of the two conditions by \cite[Proposition 3.13]{SZ}.

\vskip 10pt

Under the condition, Theorem 3.5 in \cite{G17} tells us that $A$ is Gorenstein if and only if so are $B$ and $C$.
\end{proof}

\end{document}

%% file: rgb.tex
\definecolor{AliceBlue}{rgb}{0.94,0.97,1.00}
\definecolor{AntiqueWhite1}{rgb}{1.00,0.94,0.86}
\definecolor{AntiqueWhite2}{rgb}{0.93,0.87,0.80}
\definecolor{AntiqueWhite3}{rgb}{0.80,0.75,0.69}
\definecolor{AntiqueWhite4}{rgb}{0.55,0.51,0.47}
\definecolor{AntiqueWhite}{rgb}{0.98,0.92,0.84}
\definecolor{BlanchedAlmond}{rgb}{1.00,0.92,0.80}
\definecolor{BlueViolet}{rgb}{0.54,0.17,0.89}
\definecolor{CadetBlue1}{rgb}{0.60,0.96,1.00}
\definecolor{CadetBlue2}{rgb}{0.56,0.90,0.93}
\definecolor{CadetBlue3}{rgb}{0.48,0.77,0.80}
\definecolor{CadetBlue4}{rgb}{0.33,0.53,0.55}
\definecolor{CadetBlue}{rgb}{0.37,0.62,0.63}
\definecolor{CornflowerBlue}{rgb}{0.39,0.58,0.93}
\definecolor{DarkBlue}{rgb}{0.00,0.00,0.55}
\definecolor{DarkCyan}{rgb}{0.00,0.55,0.55}
\definecolor{DarkGoldenrod1}{rgb}{1.00,0.73,0.06}
\definecolor{DarkGoldenrod2}{rgb}{0.93,0.68,0.05}
\definecolor{DarkGoldenrod3}{rgb}{0.80,0.58,0.05}
\definecolor{DarkGoldenrod4}{rgb}{0.55,0.40,0.03}
\definecolor{DarkGoldenrod}{rgb}{0.72,0.53,0.04}
\definecolor{DarkGray}{rgb}{0.66,0.66,0.66}
\definecolor{DarkGreen}{rgb}{0.00,0.39,0.00}
\definecolor{DarkGrey}{rgb}{0.66,0.66,0.66}
\definecolor{DarkKhaki}{rgb}{0.74,0.72,0.42}
\definecolor{DarkMagenta}{rgb}{0.55,0.00,0.55}
\definecolor{DarkOliveGreen1}{rgb}{0.79,1.00,0.44}
\definecolor{DarkOliveGreen2}{rgb}{0.74,0.93,0.41}
\definecolor{DarkOliveGreen3}{rgb}{0.64,0.80,0.35}
\definecolor{DarkOliveGreen4}{rgb}{0.43,0.55,0.24}
\definecolor{DarkOliveGreen}{rgb}{0.33,0.42,0.18}
\definecolor{DarkOrange1}{rgb}{1.00,0.50,0.00}
\definecolor{DarkOrange2}{rgb}{0.93,0.46,0.00}
\definecolor{DarkOrange3}{rgb}{0.80,0.40,0.00}
\definecolor{DarkOrange4}{rgb}{0.55,0.27,0.00}
\definecolor{DarkOrange}{rgb}{1.00,0.55,0.00}
\definecolor{DarkOrchid1}{rgb}{0.75,0.24,1.00}
\definecolor{DarkOrchid2}{rgb}{0.70,0.23,0.93}
\definecolor{DarkOrchid3}{rgb}{0.60,0.20,0.80}
\definecolor{DarkOrchid4}{rgb}{0.41,0.13,0.55}
\definecolor{DarkOrchid}{rgb}{0.60,0.20,0.80}
\definecolor{DarkRed}{rgb}{0.55,0.00,0.00}
\definecolor{DarkSalmon}{rgb}{0.91,0.59,0.48}
\definecolor{DarkSeaGreen1}{rgb}{0.76,1.00,0.76}
\definecolor{DarkSeaGreen2}{rgb}{0.71,0.93,0.71}
\definecolor{DarkSeaGreen3}{rgb}{0.61,0.80,0.61}
\definecolor{DarkSeaGreen4}{rgb}{0.41,0.55,0.41}
\definecolor{DarkSeaGreen}{rgb}{0.56,0.74,0.56}
\definecolor{DarkSlateBlue}{rgb}{0.28,0.24,0.55}
\definecolor{DarkSlateGray1}{rgb}{0.59,1.00,1.00}
\definecolor{DarkSlateGray2}{rgb}{0.55,0.93,0.93}
\definecolor{DarkSlateGray3}{rgb}{0.47,0.80,0.80}
\definecolor{DarkSlateGray4}{rgb}{0.32,0.55,0.55}
\definecolor{DarkSlateGray}{rgb}{0.18,0.31,0.31}
\definecolor{DarkSlateGrey}{rgb}{0.18,0.31,0.31}
\definecolor{DarkTurquoise}{rgb}{0.00,0.81,0.82}
\definecolor{DarkViolet}{rgb}{0.58,0.00,0.83}
\definecolor{DeepPink1}{rgb}{1.00,0.08,0.58}
\definecolor{DeepPink2}{rgb}{0.93,0.07,0.54}
\definecolor{DeepPink3}{rgb}{0.80,0.06,0.46}
\definecolor{DeepPink4}{rgb}{0.55,0.04,0.31}
\definecolor{DeepPink}{rgb}{1.00,0.08,0.58}
\definecolor{DeepSkyBlue1}{rgb}{0.00,0.75,1.00}
\definecolor{DeepSkyBlue2}{rgb}{0.00,0.70,0.93}
\definecolor{DeepSkyBlue3}{rgb}{0.00,0.60,0.80}
\definecolor{DeepSkyBlue4}{rgb}{0.00,0.41,0.55}
\definecolor{DeepSkyBlue}{rgb}{0.00,0.75,1.00}
\definecolor{DimGray}{rgb}{0.41,0.41,0.41}
\definecolor{DimGrey}{rgb}{0.41,0.41,0.41}
\definecolor{DodgerBlue1}{rgb}{0.12,0.56,1.00}
\definecolor{DodgerBlue2}{rgb}{0.11,0.53,0.93}
\definecolor{DodgerBlue3}{rgb}{0.09,0.45,0.80}
\definecolor{DodgerBlue4}{rgb}{0.06,0.31,0.55}
\definecolor{DodgerBlue}{rgb}{0.12,0.56,1.00}
\definecolor{FloralWhite}{rgb}{1.00,0.98,0.94}
\definecolor{ForestGreen}{rgb}{0.13,0.55,0.13}
\definecolor{GhostWhite}{rgb}{0.97,0.97,1.00}
\definecolor{GreenYellow}{rgb}{0.68,1.00,0.18}
\definecolor{HotPink1}{rgb}{1.00,0.43,0.71}
\definecolor{HotPink2}{rgb}{0.93,0.42,0.65}
\definecolor{HotPink3}{rgb}{0.80,0.38,0.56}
\definecolor{HotPink4}{rgb}{0.55,0.23,0.38}
\definecolor{HotPink}{rgb}{1.00,0.41,0.71}
\definecolor{IndianRed1}{rgb}{1.00,0.42,0.42}
\definecolor{IndianRed2}{rgb}{0.93,0.39,0.39}
\definecolor{IndianRed3}{rgb}{0.80,0.33,0.33}
\definecolor{IndianRed4}{rgb}{0.55,0.23,0.23}
\definecolor{IndianRed}{rgb}{0.80,0.36,0.36}
\definecolor{LavenderBlush1}{rgb}{1.00,0.94,0.96}
\definecolor{LavenderBlush2}{rgb}{0.93,0.88,0.90}
\definecolor{LavenderBlush3}{rgb}{0.80,0.76,0.77}
\definecolor{LavenderBlush4}{rgb}{0.55,0.51,0.53}
\definecolor{LavenderBlush}{rgb}{1.00,0.94,0.96}
\definecolor{LawnGreen}{rgb}{0.49,0.99,0.00}
\definecolor{LemonChiffon1}{rgb}{1.00,0.98,0.80}
\definecolor{LemonChiffon2}{rgb}{0.93,0.91,0.75}
\definecolor{LemonChiffon3}{rgb}{0.80,0.79,0.65}
\definecolor{LemonChiffon4}{rgb}{0.55,0.54,0.44}
\definecolor{LemonChiffon}{rgb}{1.00,0.98,0.80}
\definecolor{LightBlue1}{rgb}{0.75,0.94,1.00}
\definecolor{LightBlue2}{rgb}{0.70,0.87,0.93}
\definecolor{LightBlue3}{rgb}{0.60,0.75,0.80}
\definecolor{LightBlue4}{rgb}{0.41,0.51,0.55}
\definecolor{LightBlue}{rgb}{0.68,0.85,0.90}
\definecolor{LightCoral}{rgb}{0.94,0.50,0.50}
\definecolor{LightCyan1}{rgb}{0.88,1.00,1.00}
\definecolor{LightCyan2}{rgb}{0.82,0.93,0.93}
\definecolor{LightCyan3}{rgb}{0.71,0.80,0.80}
\definecolor{LightCyan4}{rgb}{0.48,0.55,0.55}
\definecolor{LightCyan}{rgb}{0.88,1.00,1.00}
\definecolor{LightGoldenrod1}{rgb}{1.00,0.93,0.55}
\definecolor{LightGoldenrod2}{rgb}{0.93,0.86,0.51}
\definecolor{LightGoldenrod3}{rgb}{0.80,0.75,0.44}
\definecolor{LightGoldenrod4}{rgb}{0.55,0.51,0.30}
\definecolor{LightGoldenrodYellow}{rgb}{0.98,0.98,0.82}
\definecolor{LightGoldenrod}{rgb}{0.93,0.87,0.51}
\definecolor{LightGray}{rgb}{0.83,0.83,0.83}
\definecolor{LightGreen}{rgb}{0.56,0.93,0.56}
\definecolor{LightGrey}{rgb}{0.83,0.83,0.83}
\definecolor{LightPink1}{rgb}{1.00,0.68,0.73}
\definecolor{LightPink2}{rgb}{0.93,0.64,0.68}
\definecolor{LightPink3}{rgb}{0.80,0.55,0.58}
\definecolor{LightPink4}{rgb}{0.55,0.37,0.40}
\definecolor{LightPink}{rgb}{1.00,0.71,0.76}
\definecolor{LightSalmon1}{rgb}{1.00,0.63,0.48}
\definecolor{LightSalmon2}{rgb}{0.93,0.58,0.45}
\definecolor{LightSalmon3}{rgb}{0.80,0.51,0.38}
\definecolor{LightSalmon4}{rgb}{0.55,0.34,0.26}
\definecolor{LightSalmon}{rgb}{1.00,0.63,0.48}
\definecolor{LightSeaGreen}{rgb}{0.13,0.70,0.67}
\definecolor{LightSkyBlue1}{rgb}{0.69,0.89,1.00}
\definecolor{LightSkyBlue2}{rgb}{0.64,0.83,0.93}
\definecolor{LightSkyBlue3}{rgb}{0.55,0.71,0.80}
\definecolor{LightSkyBlue4}{rgb}{0.38,0.48,0.55}
\definecolor{LightSkyBlue}{rgb}{0.53,0.81,0.98}
\definecolor{LightSlateBlue}{rgb}{0.52,0.44,1.00}
\definecolor{LightSlateGray}{rgb}{0.47,0.53,0.60}
\definecolor{LightSlateGrey}{rgb}{0.47,0.53,0.60}
\definecolor{LightSteelBlue1}{rgb}{0.79,0.88,1.00}
\definecolor{LightSteelBlue2}{rgb}{0.74,0.82,0.93}
\definecolor{LightSteelBlue3}{rgb}{0.64,0.71,0.80}
\definecolor{LightSteelBlue4}{rgb}{0.43,0.48,0.55}
\definecolor{LightSteelBlue}{rgb}{0.69,0.77,0.87}
\definecolor{LightYellow1}{rgb}{1.00,1.00,0.88}
\definecolor{LightYellow2}{rgb}{0.93,0.93,0.82}
\definecolor{LightYellow3}{rgb}{0.80,0.80,0.71}
\definecolor{LightYellow4}{rgb}{0.55,0.55,0.48}
\definecolor{LightYellow}{rgb}{1.00,1.00,0.88}
\definecolor{LimeGreen}{rgb}{0.20,0.80,0.20}
\definecolor{MediumAquamarine}{rgb}{0.40,0.80,0.67}
\definecolor{MediumBlue}{rgb}{0.00,0.00,0.80}
\definecolor{MediumOrchid1}{rgb}{0.88,0.40,1.00}
\definecolor{MediumOrchid2}{rgb}{0.82,0.37,0.93}
\definecolor{MediumOrchid3}{rgb}{0.71,0.32,0.80}
\definecolor{MediumOrchid4}{rgb}{0.48,0.22,0.55}
\definecolor{MediumOrchid}{rgb}{0.73,0.33,0.83}
\definecolor{MediumPurple1}{rgb}{0.67,0.51,1.00}
\definecolor{MediumPurple2}{rgb}{0.62,0.47,0.93}
\definecolor{MediumPurple3}{rgb}{0.54,0.41,0.80}
\definecolor{MediumPurple4}{rgb}{0.36,0.28,0.55}
\definecolor{MediumPurple}{rgb}{0.58,0.44,0.86}
\definecolor{MediumSeaGreen}{rgb}{0.24,0.70,0.44}
\definecolor{MediumSlateBlue}{rgb}{0.48,0.41,0.93}
\definecolor{MediumSpringGreen}{rgb}{0.00,0.98,0.60}
\definecolor{MediumTurquoise}{rgb}{0.28,0.82,0.80}
\definecolor{MediumVioletRed}{rgb}{0.78,0.08,0.52}
\definecolor{MidnightBlue}{rgb}{0.10,0.10,0.44}
\definecolor{MintCream}{rgb}{0.96,1.00,0.98}
\definecolor{MistyRose1}{rgb}{1.00,0.89,0.88}
\definecolor{MistyRose2}{rgb}{0.93,0.84,0.82}
\definecolor{MistyRose3}{rgb}{0.80,0.72,0.71}
\definecolor{MistyRose4}{rgb}{0.55,0.49,0.48}
\definecolor{MistyRose}{rgb}{1.00,0.89,0.88}
\definecolor{NavajoWhite1}{rgb}{1.00,0.87,0.68}
\definecolor{NavajoWhite2}{rgb}{0.93,0.81,0.63}
\definecolor{NavajoWhite3}{rgb}{0.80,0.70,0.55}
\definecolor{NavajoWhite4}{rgb}{0.55,0.47,0.37}
\definecolor{NavajoWhite}{rgb}{1.00,0.87,0.68}
\definecolor{NavyBlue}{rgb}{0.00,0.00,0.50}
\definecolor{OldLace}{rgb}{0.99,0.96,0.90}
\definecolor{OliveDrab1}{rgb}{0.75,1.00,0.24}
\definecolor{OliveDrab2}{rgb}{0.70,0.93,0.23}
\definecolor{OliveDrab3}{rgb}{0.60,0.80,0.20}
\definecolor{OliveDrab4}{rgb}{0.41,0.55,0.13}
\definecolor{OliveDrab}{rgb}{0.42,0.56,0.14}
\definecolor{OrangeRed1}{rgb}{1.00,0.27,0.00}
\definecolor{OrangeRed2}{rgb}{0.93,0.25,0.00}
\definecolor{OrangeRed3}{rgb}{0.80,0.22,0.00}
\definecolor{OrangeRed4}{rgb}{0.55,0.15,0.00}
\definecolor{OrangeRed}{rgb}{1.00,0.27,0.00}
\definecolor{PaleGoldenrod}{rgb}{0.93,0.91,0.67}
\definecolor{PaleGreen1}{rgb}{0.60,1.00,0.60}
\definecolor{PaleGreen2}{rgb}{0.56,0.93,0.56}
\definecolor{PaleGreen3}{rgb}{0.49,0.80,0.49}
\definecolor{PaleGreen4}{rgb}{0.33,0.55,0.33}
\definecolor{PaleGreen}{rgb}{0.60,0.98,0.60}
\definecolor{PaleTurquoise1}{rgb}{0.73,1.00,1.00}
\definecolor{PaleTurquoise2}{rgb}{0.68,0.93,0.93}
\definecolor{PaleTurquoise3}{rgb}{0.59,0.80,0.80}
\definecolor{PaleTurquoise4}{rgb}{0.40,0.55,0.55}
\definecolor{PaleTurquoise}{rgb}{0.69,0.93,0.93}
\definecolor{PaleVioletRed1}{rgb}{1.00,0.51,0.67}
\definecolor{PaleVioletRed2}{rgb}{0.93,0.47,0.62}
\definecolor{PaleVioletRed3}{rgb}{0.80,0.41,0.54}
\definecolor{PaleVioletRed4}{rgb}{0.55,0.28,0.36}
\definecolor{PaleVioletRed}{rgb}{0.86,0.44,0.58}
\definecolor{PapayaWhip}{rgb}{1.00,0.94,0.84}
\definecolor{PeachPuff1}{rgb}{1.00,0.85,0.73}
\definecolor{PeachPuff2}{rgb}{0.93,0.80,0.68}
\definecolor{PeachPuff3}{rgb}{0.80,0.69,0.58}
\definecolor{PeachPuff4}{rgb}{0.55,0.47,0.40}
\definecolor{PeachPuff}{rgb}{1.00,0.85,0.73}
\definecolor{PowderBlue}{rgb}{0.69,0.88,0.90}
\definecolor{RosyBrown1}{rgb}{1.00,0.76,0.76}
\definecolor{RosyBrown2}{rgb}{0.93,0.71,0.71}
\definecolor{RosyBrown3}{rgb}{0.80,0.61,0.61}
\definecolor{RosyBrown4}{rgb}{0.55,0.41,0.41}
\definecolor{RosyBrown}{rgb}{0.74,0.56,0.56}
\definecolor{RoyalBlue1}{rgb}{0.28,0.46,1.00}
\definecolor{RoyalBlue2}{rgb}{0.26,0.43,0.93}
\definecolor{RoyalBlue3}{rgb}{0.23,0.37,0.80}
\definecolor{RoyalBlue4}{rgb}{0.15,0.25,0.55}
\definecolor{RoyalBlue}{rgb}{0.25,0.41,0.88}
\definecolor{SaddleBrown}{rgb}{0.55,0.27,0.07}
\definecolor{SandyBrown}{rgb}{0.96,0.64,0.38}
\definecolor{SeaGreen1}{rgb}{0.33,1.00,0.62}
\definecolor{SeaGreen2}{rgb}{0.31,0.93,0.58}
\definecolor{SeaGreen3}{rgb}{0.26,0.80,0.50}
\definecolor{SeaGreen4}{rgb}{0.18,0.55,0.34}
\definecolor{SeaGreen}{rgb}{0.18,0.55,0.34}
\definecolor{SkyBlue1}{rgb}{0.53,0.81,1.00}
\definecolor{SkyBlue2}{rgb}{0.49,0.75,0.93}
\definecolor{SkyBlue3}{rgb}{0.42,0.65,0.80}
\definecolor{SkyBlue4}{rgb}{0.29,0.44,0.55}
\definecolor{SkyBlue}{rgb}{0.53,0.81,0.92}
\definecolor{SlateBlue1}{rgb}{0.51,0.44,1.00}
\definecolor{SlateBlue2}{rgb}{0.48,0.40,0.93}
\definecolor{SlateBlue3}{rgb}{0.41,0.35,0.80}
\definecolor{SlateBlue4}{rgb}{0.28,0.24,0.55}
\definecolor{SlateBlue}{rgb}{0.42,0.35,0.80}
\definecolor{SlateGray1}{rgb}{0.78,0.89,1.00}
\definecolor{SlateGray2}{rgb}{0.73,0.83,0.93}
\definecolor{SlateGray3}{rgb}{0.62,0.71,0.80}
\definecolor{SlateGray4}{rgb}{0.42,0.48,0.55}
\definecolor{SlateGray}{rgb}{0.44,0.50,0.56}
\definecolor{SlateGrey}{rgb}{0.44,0.50,0.56}
\definecolor{SpringGreen1}{rgb}{0.00,1.00,0.50}
\definecolor{SpringGreen2}{rgb}{0.00,0.93,0.46}
\definecolor{SpringGreen3}{rgb}{0.00,0.80,0.40}
\definecolor{SpringGreen4}{rgb}{0.00,0.55,0.27}
\definecolor{SpringGreen}{rgb}{0.00,1.00,0.50}
\definecolor{SteelBlue1}{rgb}{0.39,0.72,1.00}
\definecolor{SteelBlue2}{rgb}{0.36,0.67,0.93}
\definecolor{SteelBlue3}{rgb}{0.31,0.58,0.80}
\definecolor{SteelBlue4}{rgb}{0.21,0.39,0.55}
\definecolor{SteelBlue}{rgb}{0.27,0.51,0.71}
\definecolor{VioletRed1}{rgb}{1.00,0.24,0.59}
\definecolor{VioletRed2}{rgb}{0.93,0.23,0.55}
\definecolor{VioletRed3}{rgb}{0.80,0.20,0.47}
\definecolor{VioletRed4}{rgb}{0.55,0.13,0.32}
\definecolor{VioletRed}{rgb}{0.82,0.13,0.56}
\definecolor{WhiteSmoke}{rgb}{0.96,0.96,0.96}
\definecolor{YellowGreen}{rgb}{0.60,0.80,0.20}
\definecolor{aliceblue}{rgb}{0.94,0.97,1.00}
\definecolor{antiquewhite}{rgb}{0.98,0.92,0.84}
\definecolor{aquamarine1}{rgb}{0.50,1.00,0.83}
\definecolor{aquamarine2}{rgb}{0.46,0.93,0.78}
\definecolor{aquamarine3}{rgb}{0.40,0.80,0.67}
\definecolor{aquamarine4}{rgb}{0.27,0.55,0.45}
\definecolor{aquamarine}{rgb}{0.50,1.00,0.83}
\definecolor{azure1}{rgb}{0.94,1.00,1.00}
\definecolor{azure2}{rgb}{0.88,0.93,0.93}
\definecolor{azure3}{rgb}{0.76,0.80,0.80}
\definecolor{azure4}{rgb}{0.51,0.55,0.55}
\definecolor{azure}{rgb}{0.94,1.00,1.00}
\definecolor{beige}{rgb}{0.96,0.96,0.86}
\definecolor{bisque1}{rgb}{1.00,0.89,0.77}
\definecolor{bisque2}{rgb}{0.93,0.84,0.72}
\definecolor{bisque3}{rgb}{0.80,0.72,0.62}
\definecolor{bisque4}{rgb}{0.55,0.49,0.42}
\definecolor{bisque}{rgb}{1.00,0.89,0.77}
\definecolor{black}{rgb}{0.00,0.00,0.00}
\definecolor{blanchedalmond}{rgb}{1.00,0.92,0.80}
\definecolor{blue1}{rgb}{0.00,0.00,1.00}
\definecolor{blue2}{rgb}{0.00,0.00,0.93}
\definecolor{blue3}{rgb}{0.00,0.00,0.80}
\definecolor{blue4}{rgb}{0.00,0.00,0.55}
\definecolor{blueviolet}{rgb}{0.54,0.17,0.89}
\definecolor{blue}{rgb}{0.00,0.00,1.00}
\definecolor{brown1}{rgb}{1.00,0.25,0.25}
\definecolor{brown2}{rgb}{0.93,0.23,0.23}
\definecolor{brown3}{rgb}{0.80,0.20,0.20}
\definecolor{brown4}{rgb}{0.55,0.14,0.14}
\definecolor{brown}{rgb}{0.65,0.16,0.16}
\definecolor{burlywood1}{rgb}{1.00,0.83,0.61}
\definecolor{burlywood2}{rgb}{0.93,0.77,0.57}
\definecolor{burlywood3}{rgb}{0.80,0.67,0.49}
\definecolor{burlywood4}{rgb}{0.55,0.45,0.33}
\definecolor{burlywood}{rgb}{0.87,0.72,0.53}
\definecolor{cadetblue}{rgb}{0.37,0.62,0.63}
\definecolor{chartreuse1}{rgb}{0.50,1.00,0.00}
\definecolor{chartreuse2}{rgb}{0.46,0.93,0.00}
\definecolor{chartreuse3}{rgb}{0.40,0.80,0.00}
\definecolor{chartreuse4}{rgb}{0.27,0.55,0.00}
\definecolor{chartreuse}{rgb}{0.50,1.00,0.00}
\definecolor{chocolate1}{rgb}{1.00,0.50,0.14}
\definecolor{chocolate2}{rgb}{0.93,0.46,0.13}
\definecolor{chocolate3}{rgb}{0.80,0.40,0.11}
\definecolor{chocolate4}{rgb}{0.55,0.27,0.07}
\definecolor{chocolate}{rgb}{0.82,0.41,0.12}
\definecolor{coral1}{rgb}{1.00,0.45,0.34}
\definecolor{coral2}{rgb}{0.93,0.42,0.31}
\definecolor{coral3}{rgb}{0.80,0.36,0.27}
\definecolor{coral4}{rgb}{0.55,0.24,0.18}
\definecolor{coral}{rgb}{1.00,0.50,0.31}
\definecolor{cornflowerblue}{rgb}{0.39,0.58,0.93}
\definecolor{cornsilk1}{rgb}{1.00,0.97,0.86}
\definecolor{cornsilk2}{rgb}{0.93,0.91,0.80}
\definecolor{cornsilk3}{rgb}{0.80,0.78,0.69}
\definecolor{cornsilk4}{rgb}{0.55,0.53,0.47}
\definecolor{cornsilk}{rgb}{1.00,0.97,0.86}
\definecolor{cyan1}{rgb}{0.00,1.00,1.00}
\definecolor{cyan2}{rgb}{0.00,0.93,0.93}
\definecolor{cyan3}{rgb}{0.00,0.80,0.80}
\definecolor{cyan4}{rgb}{0.00,0.55,0.55}
\definecolor{cyan}{rgb}{0.00,1.00,1.00}
\definecolor{darkblue}{rgb}{0.00,0.00,0.55}
\definecolor{darkcyan}{rgb}{0.00,0.55,0.55}
\definecolor{darkgoldenrod}{rgb}{0.72,0.53,0.04}
\definecolor{darkgray}{rgb}{0.66,0.66,0.66}
\definecolor{darkgreen}{rgb}{0.00,0.39,0.00}
\definecolor{darkgrey}{rgb}{0.66,0.66,0.66}
\definecolor{darkkhaki}{rgb}{0.74,0.72,0.42}
\definecolor{darkmagenta}{rgb}{0.55,0.00,0.55}
\definecolor{darkolive}{rgb}{0.33,0.42,0.18}
\definecolor{darkorange}{rgb}{1.00,0.55,0.00}
\definecolor{darkorchid}{rgb}{0.60,0.20,0.80}
\definecolor{darkred}{rgb}{0.55,0.00,0.00}
\definecolor{darksalmon}{rgb}{0.91,0.59,0.48}
\definecolor{darksea}{rgb}{0.56,0.74,0.56}
\definecolor{darkslate}{rgb}{0.18,0.31,0.31}
\definecolor{darkslate}{rgb}{0.18,0.31,0.31}
\definecolor{darkslate}{rgb}{0.28,0.24,0.55}
\definecolor{darkturquoise}{rgb}{0.00,0.81,0.82}
\definecolor{darkviolet}{rgb}{0.58,0.00,0.83}
\definecolor{deeppink}{rgb}{1.00,0.08,0.58}
\definecolor{deepsky}{rgb}{0.00,0.75,1.00}
\definecolor{dimgray}{rgb}{0.41,0.41,0.41}
\definecolor{dimgrey}{rgb}{0.41,0.41,0.41}
\definecolor{dodgerblue}{rgb}{0.12,0.56,1.00}
\definecolor{firebrick1}{rgb}{1.00,0.19,0.19}
\definecolor{firebrick2}{rgb}{0.93,0.17,0.17}
\definecolor{firebrick3}{rgb}{0.80,0.15,0.15}
\definecolor{firebrick4}{rgb}{0.55,0.10,0.10}
\definecolor{firebrick}{rgb}{0.70,0.13,0.13}
\definecolor{floralwhite}{rgb}{1.00,0.98,0.94}
\definecolor{forestgreen}{rgb}{0.13,0.55,0.13}
\definecolor{gainsboro}{rgb}{0.86,0.86,0.86}
\definecolor{ghostwhite}{rgb}{0.97,0.97,1.00}
\definecolor{gold1}{rgb}{1.00,0.84,0.00}
\definecolor{gold2}{rgb}{0.93,0.79,0.00}
\definecolor{gold3}{rgb}{0.80,0.68,0.00}
\definecolor{gold4}{rgb}{0.55,0.46,0.00}
\definecolor{goldenrod1}{rgb}{1.00,0.76,0.15}
\definecolor{goldenrod2}{rgb}{0.93,0.71,0.13}
\definecolor{goldenrod3}{rgb}{0.80,0.61,0.11}
\definecolor{goldenrod4}{rgb}{0.55,0.41,0.08}
\definecolor{goldenrod}{rgb}{0.85,0.65,0.13}
\definecolor{gold}{rgb}{1.00,0.84,0.00}
\definecolor{gray0}{rgb}{0.00,0.00,0.00}
\definecolor{gray100}{rgb}{1.00,1.00,1.00}
\definecolor{gray10}{rgb}{0.10,0.10,0.10}
\definecolor{gray11}{rgb}{0.11,0.11,0.11}
\definecolor{gray12}{rgb}{0.12,0.12,0.12}
\definecolor{gray13}{rgb}{0.13,0.13,0.13}
\definecolor{gray14}{rgb}{0.14,0.14,0.14}
\definecolor{gray15}{rgb}{0.15,0.15,0.15}
\definecolor{gray16}{rgb}{0.16,0.16,0.16}
\definecolor{gray17}{rgb}{0.17,0.17,0.17}
\definecolor{gray18}{rgb}{0.18,0.18,0.18}
\definecolor{gray19}{rgb}{0.19,0.19,0.19}
\definecolor{gray1}{rgb}{0.01,0.01,0.01}
\definecolor{gray20}{rgb}{0.20,0.20,0.20}
\definecolor{gray21}{rgb}{0.21,0.21,0.21}
\definecolor{gray22}{rgb}{0.22,0.22,0.22}
\definecolor{gray23}{rgb}{0.23,0.23,0.23}
\definecolor{gray24}{rgb}{0.24,0.24,0.24}
\definecolor{gray25}{rgb}{0.25,0.25,0.25}
\definecolor{gray26}{rgb}{0.26,0.26,0.26}
\definecolor{gray27}{rgb}{0.27,0.27,0.27}
\definecolor{gray28}{rgb}{0.28,0.28,0.28}
\definecolor{gray29}{rgb}{0.29,0.29,0.29}
\definecolor{gray2}{rgb}{0.02,0.02,0.02}
\definecolor{gray30}{rgb}{0.30,0.30,0.30}
\definecolor{gray31}{rgb}{0.31,0.31,0.31}
\definecolor{gray32}{rgb}{0.32,0.32,0.32}
\definecolor{gray33}{rgb}{0.33,0.33,0.33}
\definecolor{gray34}{rgb}{0.34,0.34,0.34}
\definecolor{gray35}{rgb}{0.35,0.35,0.35}
\definecolor{gray36}{rgb}{0.36,0.36,0.36}
\definecolor{gray37}{rgb}{0.37,0.37,0.37}
\definecolor{gray38}{rgb}{0.38,0.38,0.38}
\definecolor{gray39}{rgb}{0.39,0.39,0.39}
\definecolor{gray3}{rgb}{0.03,0.03,0.03}
\definecolor{gray40}{rgb}{0.40,0.40,0.40}
\definecolor{gray41}{rgb}{0.41,0.41,0.41}
\definecolor{gray42}{rgb}{0.42,0.42,0.42}
\definecolor{gray43}{rgb}{0.43,0.43,0.43}
\definecolor{gray44}{rgb}{0.44,0.44,0.44}
\definecolor{gray45}{rgb}{0.45,0.45,0.45}
\definecolor{gray46}{rgb}{0.46,0.46,0.46}
\definecolor{gray47}{rgb}{0.47,0.47,0.47}
\definecolor{gray48}{rgb}{0.48,0.48,0.48}
\definecolor{gray49}{rgb}{0.49,0.49,0.49}
\definecolor{gray4}{rgb}{0.04,0.04,0.04}
\definecolor{gray50}{rgb}{0.50,0.50,0.50}
\definecolor{gray51}{rgb}{0.51,0.51,0.51}
\definecolor{gray52}{rgb}{0.52,0.52,0.52}
\definecolor{gray53}{rgb}{0.53,0.53,0.53}
\definecolor{gray54}{rgb}{0.54,0.54,0.54}
\definecolor{gray55}{rgb}{0.55,0.55,0.55}
\definecolor{gray56}{rgb}{0.56,0.56,0.56}
\definecolor{gray57}{rgb}{0.57,0.57,0.57}
\definecolor{gray58}{rgb}{0.58,0.58,0.58}
\definecolor{gray59}{rgb}{0.59,0.59,0.59}
\definecolor{gray5}{rgb}{0.05,0.05,0.05}
\definecolor{gray60}{rgb}{0.60,0.60,0.60}
\definecolor{gray61}{rgb}{0.61,0.61,0.61}
\definecolor{gray62}{rgb}{0.62,0.62,0.62}
\definecolor{gray63}{rgb}{0.63,0.63,0.63}
\definecolor{gray64}{rgb}{0.64,0.64,0.64}
\definecolor{gray65}{rgb}{0.65,0.65,0.65}
\definecolor{gray66}{rgb}{0.66,0.66,0.66}
\definecolor{gray67}{rgb}{0.67,0.67,0.67}
\definecolor{gray68}{rgb}{0.68,0.68,0.68}
\definecolor{gray69}{rgb}{0.69,0.69,0.69}
\definecolor{gray6}{rgb}{0.06,0.06,0.06}
\definecolor{gray70}{rgb}{0.70,0.70,0.70}
\definecolor{gray71}{rgb}{0.71,0.71,0.71}
\definecolor{gray72}{rgb}{0.72,0.72,0.72}
\definecolor{gray73}{rgb}{0.73,0.73,0.73}
\definecolor{gray74}{rgb}{0.74,0.74,0.74}
\definecolor{gray75}{rgb}{0.75,0.75,0.75}
\definecolor{gray76}{rgb}{0.76,0.76,0.76}
\definecolor{gray77}{rgb}{0.77,0.77,0.77}
\definecolor{gray78}{rgb}{0.78,0.78,0.78}
\definecolor{gray79}{rgb}{0.79,0.79,0.79}
\definecolor{gray7}{rgb}{0.07,0.07,0.07}
\definecolor{gray80}{rgb}{0.80,0.80,0.80}
\definecolor{gray81}{rgb}{0.81,0.81,0.81}
\definecolor{gray82}{rgb}{0.82,0.82,0.82}
\definecolor{gray83}{rgb}{0.83,0.83,0.83}
\definecolor{gray84}{rgb}{0.84,0.84,0.84}
\definecolor{gray85}{rgb}{0.85,0.85,0.85}
\definecolor{gray86}{rgb}{0.86,0.86,0.86}
\definecolor{gray87}{rgb}{0.87,0.87,0.87}
\definecolor{gray88}{rgb}{0.88,0.88,0.88}
\definecolor{gray89}{rgb}{0.89,0.89,0.89}
\definecolor{gray8}{rgb}{0.08,0.08,0.08}
\definecolor{gray90}{rgb}{0.90,0.90,0.90}
\definecolor{gray91}{rgb}{0.91,0.91,0.91}
\definecolor{gray92}{rgb}{0.92,0.92,0.92}
\definecolor{gray93}{rgb}{0.93,0.93,0.93}
\definecolor{gray94}{rgb}{0.94,0.94,0.94}
\definecolor{gray95}{rgb}{0.95,0.95,0.95}
\definecolor{gray96}{rgb}{0.96,0.96,0.96}
\definecolor{gray97}{rgb}{0.97,0.97,0.97}
\definecolor{gray98}{rgb}{0.98,0.98,0.98}
\definecolor{gray99}{rgb}{0.99,0.99,0.99}
\definecolor{gray9}{rgb}{0.09,0.09,0.09}
\definecolor{gray}{rgb}{0.75,0.75,0.75}
\definecolor{green1}{rgb}{0.00,1.00,0.00}
\definecolor{green2}{rgb}{0.00,0.93,0.00}
\definecolor{green3}{rgb}{0.00,0.80,0.00}
\definecolor{green4}{rgb}{0.00,0.55,0.00}
\definecolor{greenyellow}{rgb}{0.68,1.00,0.18}
\definecolor{green}{rgb}{0.00,1.00,0.00}
\definecolor{grey0}{rgb}{0.00,0.00,0.00}
\definecolor{grey100}{rgb}{1.00,1.00,1.00}
\definecolor{grey10}{rgb}{0.10,0.10,0.10}
\definecolor{grey11}{rgb}{0.11,0.11,0.11}
\definecolor{grey12}{rgb}{0.12,0.12,0.12}
\definecolor{grey13}{rgb}{0.13,0.13,0.13}
\definecolor{grey14}{rgb}{0.14,0.14,0.14}
\definecolor{grey15}{rgb}{0.15,0.15,0.15}
\definecolor{grey16}{rgb}{0.16,0.16,0.16}
\definecolor{grey17}{rgb}{0.17,0.17,0.17}
\definecolor{grey18}{rgb}{0.18,0.18,0.18}
\definecolor{grey19}{rgb}{0.19,0.19,0.19}
\definecolor{grey1}{rgb}{0.01,0.01,0.01}
\definecolor{grey20}{rgb}{0.20,0.20,0.20}
\definecolor{grey21}{rgb}{0.21,0.21,0.21}
\definecolor{grey22}{rgb}{0.22,0.22,0.22}
\definecolor{grey23}{rgb}{0.23,0.23,0.23}
\definecolor{grey24}{rgb}{0.24,0.24,0.24}
\definecolor{grey25}{rgb}{0.25,0.25,0.25}
\definecolor{grey26}{rgb}{0.26,0.26,0.26}
\definecolor{grey27}{rgb}{0.27,0.27,0.27}
\definecolor{grey28}{rgb}{0.28,0.28,0.28}
\definecolor{grey29}{rgb}{0.29,0.29,0.29}
\definecolor{grey2}{rgb}{0.02,0.02,0.02}
\definecolor{grey30}{rgb}{0.30,0.30,0.30}
\definecolor{grey31}{rgb}{0.31,0.31,0.31}
\definecolor{grey32}{rgb}{0.32,0.32,0.32}
\definecolor{grey33}{rgb}{0.33,0.33,0.33}
\definecolor{grey34}{rgb}{0.34,0.34,0.34}
\definecolor{grey35}{rgb}{0.35,0.35,0.35}
\definecolor{grey36}{rgb}{0.36,0.36,0.36}
\definecolor{grey37}{rgb}{0.37,0.37,0.37}
\definecolor{grey38}{rgb}{0.38,0.38,0.38}
\definecolor{grey39}{rgb}{0.39,0.39,0.39}
\definecolor{grey3}{rgb}{0.03,0.03,0.03}
\definecolor{grey40}{rgb}{0.40,0.40,0.40}
\definecolor{grey41}{rgb}{0.41,0.41,0.41}
\definecolor{grey42}{rgb}{0.42,0.42,0.42}
\definecolor{grey43}{rgb}{0.43,0.43,0.43}
\definecolor{grey44}{rgb}{0.44,0.44,0.44}
\definecolor{grey45}{rgb}{0.45,0.45,0.45}
\definecolor{grey46}{rgb}{0.46,0.46,0.46}
\definecolor{grey47}{rgb}{0.47,0.47,0.47}
\definecolor{grey48}{rgb}{0.48,0.48,0.48}
\definecolor{grey49}{rgb}{0.49,0.49,0.49}
\definecolor{grey4}{rgb}{0.04,0.04,0.04}
\definecolor{grey50}{rgb}{0.50,0.50,0.50}
\definecolor{grey51}{rgb}{0.51,0.51,0.51}
\definecolor{grey52}{rgb}{0.52,0.52,0.52}
\definecolor{grey53}{rgb}{0.53,0.53,0.53}
\definecolor{grey54}{rgb}{0.54,0.54,0.54}
\definecolor{grey55}{rgb}{0.55,0.55,0.55}
\definecolor{grey56}{rgb}{0.56,0.56,0.56}
\definecolor{grey57}{rgb}{0.57,0.57,0.57}
\definecolor{grey58}{rgb}{0.58,0.58,0.58}
\definecolor{grey59}{rgb}{0.59,0.59,0.59}
\definecolor{grey5}{rgb}{0.05,0.05,0.05}
\definecolor{grey60}{rgb}{0.60,0.60,0.60}
\definecolor{grey61}{rgb}{0.61,0.61,0.61}
\definecolor{grey62}{rgb}{0.62,0.62,0.62}
\definecolor{grey63}{rgb}{0.63,0.63,0.63}
\definecolor{grey64}{rgb}{0.64,0.64,0.64}
\definecolor{grey65}{rgb}{0.65,0.65,0.65}
\definecolor{grey66}{rgb}{0.66,0.66,0.66}
\definecolor{grey67}{rgb}{0.67,0.67,0.67}
\definecolor{grey68}{rgb}{0.68,0.68,0.68}
\definecolor{grey69}{rgb}{0.69,0.69,0.69}
\definecolor{grey6}{rgb}{0.06,0.06,0.06}
\definecolor{grey70}{rgb}{0.70,0.70,0.70}
\definecolor{grey71}{rgb}{0.71,0.71,0.71}
\definecolor{grey72}{rgb}{0.72,0.72,0.72}
\definecolor{grey73}{rgb}{0.73,0.73,0.73}
\definecolor{grey74}{rgb}{0.74,0.74,0.74}
\definecolor{grey75}{rgb}{0.75,0.75,0.75}
\definecolor{grey76}{rgb}{0.76,0.76,0.76}
\definecolor{grey77}{rgb}{0.77,0.77,0.77}
\definecolor{grey78}{rgb}{0.78,0.78,0.78}
\definecolor{grey79}{rgb}{0.79,0.79,0.79}
\definecolor{grey7}{rgb}{0.07,0.07,0.07}
\definecolor{grey80}{rgb}{0.80,0.80,0.80}
\definecolor{grey81}{rgb}{0.81,0.81,0.81}
\definecolor{grey82}{rgb}{0.82,0.82,0.82}
\definecolor{grey83}{rgb}{0.83,0.83,0.83}
\definecolor{grey84}{rgb}{0.84,0.84,0.84}
\definecolor{grey85}{rgb}{0.85,0.85,0.85}
\definecolor{grey86}{rgb}{0.86,0.86,0.86}
\definecolor{grey87}{rgb}{0.87,0.87,0.87}
\definecolor{grey88}{rgb}{0.88,0.88,0.88}
\definecolor{grey89}{rgb}{0.89,0.89,0.89}
\definecolor{grey8}{rgb}{0.08,0.08,0.08}
\definecolor{grey90}{rgb}{0.90,0.90,0.90}
\definecolor{grey91}{rgb}{0.91,0.91,0.91}
\definecolor{grey92}{rgb}{0.92,0.92,0.92}
\definecolor{grey93}{rgb}{0.93,0.93,0.93}
\definecolor{grey94}{rgb}{0.94,0.94,0.94}
\definecolor{grey95}{rgb}{0.95,0.95,0.95}
\definecolor{grey96}{rgb}{0.96,0.96,0.96}
\definecolor{grey97}{rgb}{0.97,0.97,0.97}
\definecolor{grey98}{rgb}{0.98,0.98,0.98}
\definecolor{grey99}{rgb}{0.99,0.99,0.99}
\definecolor{grey9}{rgb}{0.09,0.09,0.09}
\definecolor{grey}{rgb}{0.75,0.75,0.75}
\definecolor{honeydew1}{rgb}{0.94,1.00,0.94}
\definecolor{honeydew2}{rgb}{0.88,0.93,0.88}
\definecolor{honeydew3}{rgb}{0.76,0.80,0.76}
\definecolor{honeydew4}{rgb}{0.51,0.55,0.51}
\definecolor{honeydew}{rgb}{0.94,1.00,0.94}
\definecolor{hotpink}{rgb}{1.00,0.41,0.71}
\definecolor{indianred}{rgb}{0.80,0.36,0.36}
\definecolor{ivory1}{rgb}{1.00,1.00,0.94}
\definecolor{ivory2}{rgb}{0.93,0.93,0.88}
\definecolor{ivory3}{rgb}{0.80,0.80,0.76}
\definecolor{ivory4}{rgb}{0.55,0.55,0.51}
\definecolor{ivory}{rgb}{1.00,1.00,0.94}
\definecolor{khaki1}{rgb}{1.00,0.96,0.56}
\definecolor{khaki2}{rgb}{0.93,0.90,0.52}
\definecolor{khaki3}{rgb}{0.80,0.78,0.45}
\definecolor{khaki4}{rgb}{0.55,0.53,0.31}
\definecolor{khaki}{rgb}{0.94,0.90,0.55}
\definecolor{lavenderblush}{rgb}{1.00,0.94,0.96}
\definecolor{lavender}{rgb}{0.90,0.90,0.98}
\definecolor{lawngreen}{rgb}{0.49,0.99,0.00}
\definecolor{lemonchiffon}{rgb}{1.00,0.98,0.80}
\definecolor{lightblue}{rgb}{0.68,0.85,0.90}
\definecolor{lightcoral}{rgb}{0.94,0.50,0.50}
\definecolor{lightcyan}{rgb}{0.88,1.00,1.00}
\definecolor{lightgoldenrod}{rgb}{0.93,0.87,0.51}
\definecolor{lightgoldenrod}{rgb}{0.98,0.98,0.82}
\definecolor{lightgray}{rgb}{0.83,0.83,0.83}
\definecolor{lightgreen}{rgb}{0.56,0.93,0.56}
\definecolor{lightgrey}{rgb}{0.83,0.83,0.83}
\definecolor{lightpink}{rgb}{1.00,0.71,0.76}
\definecolor{lightsalmon}{rgb}{1.00,0.63,0.48}
\definecolor{lightsea}{rgb}{0.13,0.70,0.67}
\definecolor{lightsky}{rgb}{0.53,0.81,0.98}
\definecolor{lightslate}{rgb}{0.47,0.53,0.60}
\definecolor{lightslate}{rgb}{0.47,0.53,0.60}
\definecolor{lightslate}{rgb}{0.52,0.44,1.00}
\definecolor{lightsteel}{rgb}{0.69,0.77,0.87}
\definecolor{lightyellow}{rgb}{1.00,1.00,0.88}
\definecolor{limegreen}{rgb}{0.20,0.80,0.20}
\definecolor{linen}{rgb}{0.98,0.94,0.90}
\definecolor{magenta1}{rgb}{1.00,0.00,1.00}
\definecolor{magenta2}{rgb}{0.93,0.00,0.93}
\definecolor{magenta3}{rgb}{0.80,0.00,0.80}
\definecolor{magenta4}{rgb}{0.55,0.00,0.55}
\definecolor{magenta}{rgb}{1.00,0.00,1.00}
\definecolor{maroon1}{rgb}{1.00,0.20,0.70}
\definecolor{maroon2}{rgb}{0.93,0.19,0.65}
\definecolor{maroon3}{rgb}{0.80,0.16,0.56}
\definecolor{maroon4}{rgb}{0.55,0.11,0.38}
\definecolor{maroon}{rgb}{0.69,0.19,0.38}
\definecolor{mediumaquamarine}{rgb}{0.40,0.80,0.67}
\definecolor{mediumblue}{rgb}{0.00,0.00,0.80}
\definecolor{mediumorchid}{rgb}{0.73,0.33,0.83}
\definecolor{mediumpurple}{rgb}{0.58,0.44,0.86}
\definecolor{mediumsea}{rgb}{0.24,0.70,0.44}
\definecolor{mediumslate}{rgb}{0.48,0.41,0.93}
\definecolor{mediumspring}{rgb}{0.00,0.98,0.60}
\definecolor{mediumturquoise}{rgb}{0.28,0.82,0.80}
\definecolor{mediumviolet}{rgb}{0.78,0.08,0.52}
\definecolor{midnightblue}{rgb}{0.10,0.10,0.44}
\definecolor{mintcream}{rgb}{0.96,1.00,0.98}
\definecolor{mistyrose}{rgb}{1.00,0.89,0.88}
\definecolor{moccasin}{rgb}{1.00,0.89,0.71}
\definecolor{navajowhite}{rgb}{1.00,0.87,0.68}
\definecolor{navyblue}{rgb}{0.00,0.00,0.50}
\definecolor{navy}{rgb}{0.00,0.00,0.50}
\definecolor{oldlace}{rgb}{0.99,0.96,0.90}
\definecolor{olivedrab}{rgb}{0.42,0.56,0.14}
\definecolor{orange1}{rgb}{1.00,0.65,0.00}
\definecolor{orange2}{rgb}{0.93,0.60,0.00}
\definecolor{orange3}{rgb}{0.80,0.52,0.00}
\definecolor{orange4}{rgb}{0.55,0.35,0.00}
\definecolor{orangered}{rgb}{1.00,0.27,0.00}
\definecolor{orange}{rgb}{1.00,0.65,0.00}
\definecolor{orchid1}{rgb}{1.00,0.51,0.98}
\definecolor{orchid2}{rgb}{0.93,0.48,0.91}
\definecolor{orchid3}{rgb}{0.80,0.41,0.79}
\definecolor{orchid4}{rgb}{0.55,0.28,0.54}
\definecolor{orchid}{rgb}{0.85,0.44,0.84}
\definecolor{palegoldenrod}{rgb}{0.93,0.91,0.67}
\definecolor{palegreen}{rgb}{0.60,0.98,0.60}
\definecolor{paleturquoise}{rgb}{0.69,0.93,0.93}
\definecolor{paleviolet}{rgb}{0.86,0.44,0.58}
\definecolor{papayawhip}{rgb}{1.00,0.94,0.84}
\definecolor{peachpuff}{rgb}{1.00,0.85,0.73}
\definecolor{peru}{rgb}{0.80,0.52,0.25}
\definecolor{pink1}{rgb}{1.00,0.71,0.77}
\definecolor{pink2}{rgb}{0.93,0.66,0.72}
\definecolor{pink3}{rgb}{0.80,0.57,0.62}
\definecolor{pink4}{rgb}{0.55,0.39,0.42}
\definecolor{pink}{rgb}{1.00,0.75,0.80}
\definecolor{plum1}{rgb}{1.00,0.73,1.00}
\definecolor{plum2}{rgb}{0.93,0.68,0.93}
\definecolor{plum3}{rgb}{0.80,0.59,0.80}
\definecolor{plum4}{rgb}{0.55,0.40,0.55}
\definecolor{plum}{rgb}{0.87,0.63,0.87}
\definecolor{powderblue}{rgb}{0.69,0.88,0.90}
\definecolor{purple1}{rgb}{0.61,0.19,1.00}
\definecolor{purple2}{rgb}{0.57,0.17,0.93}
\definecolor{purple3}{rgb}{0.49,0.15,0.80}
\definecolor{purple4}{rgb}{0.33,0.10,0.55}
\definecolor{purple}{rgb}{0.63,0.13,0.94}
\definecolor{red1}{rgb}{1.00,0.00,0.00}
\definecolor{red2}{rgb}{0.93,0.00,0.00}
\definecolor{red3}{rgb}{0.80,0.00,0.00}
\definecolor{red4}{rgb}{0.55,0.00,0.00}
\definecolor{red}{rgb}{1.00,0.00,0.00}
\definecolor{rosybrown}{rgb}{0.74,0.56,0.56}
\definecolor{royalblue}{rgb}{0.25,0.41,0.88}
\definecolor{saddlebrown}{rgb}{0.55,0.27,0.07}
\definecolor{salmon1}{rgb}{1.00,0.55,0.41}
\definecolor{salmon2}{rgb}{0.93,0.51,0.38}
\definecolor{salmon3}{rgb}{0.80,0.44,0.33}
\definecolor{salmon4}{rgb}{0.55,0.30,0.22}
\definecolor{salmon}{rgb}{0.98,0.50,0.45}
\definecolor{sandybrown}{rgb}{0.96,0.64,0.38}
\definecolor{seagreen}{rgb}{0.18,0.55,0.34}
\definecolor{seashell1}{rgb}{1.00,0.96,0.93}
\definecolor{seashell2}{rgb}{0.93,0.90,0.87}
\definecolor{seashell3}{rgb}{0.80,0.77,0.75}
\definecolor{seashell4}{rgb}{0.55,0.53,0.51}
\definecolor{seashell}{rgb}{1.00,0.96,0.93}
\definecolor{sienna1}{rgb}{1.00,0.51,0.28}
\definecolor{sienna2}{rgb}{0.93,0.47,0.26}
\definecolor{sienna3}{rgb}{0.80,0.41,0.22}
\definecolor{sienna4}{rgb}{0.55,0.28,0.15}
\definecolor{sienna}{rgb}{0.63,0.32,0.18}
\definecolor{skyblue}{rgb}{0.53,0.81,0.92}
\definecolor{slateblue}{rgb}{0.42,0.35,0.80}
\definecolor{slategray}{rgb}{0.44,0.50,0.56}
\definecolor{slategrey}{rgb}{0.44,0.50,0.56}
\definecolor{snow1}{rgb}{1.00,0.98,0.98}
\definecolor{snow2}{rgb}{0.93,0.91,0.91}
\definecolor{snow3}{rgb}{0.80,0.79,0.79}
\definecolor{snow4}{rgb}{0.55,0.54,0.54}
\definecolor{snow}{rgb}{1.00,0.98,0.98}
\definecolor{springgreen}{rgb}{0.00,1.00,0.50}
\definecolor{steelblue}{rgb}{0.27,0.51,0.71}
\definecolor{tan1}{rgb}{1.00,0.65,0.31}
\definecolor{tan2}{rgb}{0.93,0.60,0.29}
\definecolor{tan3}{rgb}{0.80,0.52,0.25}
\definecolor{tan4}{rgb}{0.55,0.35,0.17}
\definecolor{tan}{rgb}{0.82,0.71,0.55}
\definecolor{thistle1}{rgb}{1.00,0.88,1.00}
\definecolor{thistle2}{rgb}{0.93,0.82,0.93}
\definecolor{thistle3}{rgb}{0.80,0.71,0.80}
\definecolor{thistle4}{rgb}{0.55,0.48,0.55}
\definecolor{thistle}{rgb}{0.85,0.75,0.85}
\definecolor{tomato1}{rgb}{1.00,0.39,0.28}
\definecolor{tomato2}{rgb}{0.93,0.36,0.26}
\definecolor{tomato3}{rgb}{0.80,0.31,0.22}
\definecolor{tomato4}{rgb}{0.55,0.21,0.15}
\definecolor{tomato}{rgb}{1.00,0.39,0.28}
\definecolor{turquoise1}{rgb}{0.00,0.96,1.00}
\definecolor{turquoise2}{rgb}{0.00,0.90,0.93}
\definecolor{turquoise3}{rgb}{0.00,0.77,0.80}
\definecolor{turquoise4}{rgb}{0.00,0.53,0.55}
\definecolor{turquoise}{rgb}{0.25,0.88,0.82}
\definecolor{violetred}{rgb}{0.82,0.13,0.56}
\definecolor{violet}{rgb}{0.93,0.51,0.93}
\definecolor{wheat1}{rgb}{1.00,0.91,0.73}
\definecolor{wheat2}{rgb}{0.93,0.85,0.68}
\definecolor{wheat3}{rgb}{0.80,0.73,0.59}
\definecolor{wheat4}{rgb}{0.55,0.49,0.40}
\definecolor{wheat}{rgb}{0.96,0.87,0.70}
\definecolor{whitesmoke}{rgb}{0.96,0.96,0.96}
\definecolor{white}{rgb}{1.00,1.00,1.00}
\definecolor{yellow1}{rgb}{1.00,1.00,0.00}
\definecolor{yellow2}{rgb}{0.93,0.93,0.00}
\definecolor{yellow3}{rgb}{0.80,0.80,0.00}
\definecolor{yellow4}{rgb}{0.55,0.55,0.00}
\definecolor{yellowgreen}{rgb}{0.60,0.80,0.20}
\definecolor{yellow}{rgb}{1.00,1.00,0.00}